\documentclass[graybox]{svmult}

\usepackage{helvet}         
\usepackage{courier}        
\usepackage{type1cm}        
\usepackage{makeidx}         
\usepackage{graphicx}        
\usepackage{multicol}        
\usepackage[bottom]{footmisc}
\usepackage{bm}
\usepackage{mathtools}
\usepackage{color}
\usepackage[top=4.cm, bottom=4cm, outer=4cm, inner=4cm]{geometry}

\makeindex             

\usepackage{latexsym}
\usepackage{dsfont}
\usepackage{bbm}
\usepackage{amssymb}
\usepackage{amsmath}
\usepackage{color}
\usepackage{soul}

\newtheorem{Theorem}{Theorem}[section]
\newtheorem{Lemma}[Theorem]{Lemma}

\newtheorem{Prop}[Theorem]{Proposition}
\newtheorem{Rem}[Theorem]{Remark}
\newtheorem{Exa}[Theorem]{Example}

\def\cS{\mathcal{S}}

\def\fA{\mathfrak{A}}

\def\Erw{\mathbb{E}}

\def\N{\mathbb{N}}

\def\Prob{\mathbb{P}}

\def\R{\mathbb{R}}

\def\Z{\mathbb{Z}}

\def\eps{\varepsilon}

\def\1{\vec{1}}
\def\3{{\ss}}

\def\eqdist{\stackrel{d}{=}}

\def\idist{\stackrel{d}{\to}}
\def\vag{\stackrel{v}{\to}}
\def\weakly{\stackrel{w}{\to}}
\def\todistrfd{\stackrel{f.d.d.}{\longrightarrow}}
\def\todistrD{\stackrel{d}{\Longrightarrow}}

\def\wh{\widehat}

\def\ovl{\overline}
\def\uline{\underline}

\newcommand{\ofdd}{\overset{f.d.d.}{=}}

\def\todistrfd{\stackrel{f.d.d.}{\longrightarrow}}

\allowdisplaybreaks[4] 

\begin{document}

\title*{Leader election using random walks}
\titlerunning{Leader election using random walks}
\author{Gerold Alsmeyer$^{1}$, Zakhar Kabluchko$^{1}$ and Alexander Marynych$^{1,2}$}
\institute{$^{1}$ Inst.~Math.~Statistics, Department
of Mathematics and Computer Science, University of M\"unster,
Orl\'eans-Ring 10, D-48149 M\"unster, Germany.\at
$^{2}$ Faculty of Cybernetics, Taras Shevchenko National University of Kyiv, 01601 Kyiv, Ukraine\at
\email{gerolda@math.uni-muenster.de, zakhar.kabluchko@uni-muenster.de,\at marynych@unicyb.kiev.ua}}

\maketitle

\abstract{In the classical leader election procedure all players toss coins independently and those who get tails leave the game, while those who get heads move to the next round where the procedure is repeated. We investigate a generalizion of this procedure in which the labels (positions) of the players who remain in the game are determined using an integer-valued random walk. We study the asymptotics of some relevant quantities for this model such as: the positions of the persons who remained after $n$ rounds; the total number of rounds until all the persons among $1,2,\ldots,M$ leave the game; and the number of players among $1,2,\ldots,M$ who survived the first $n$ rounds. Our results lead to some interesting connection with Galton-Watson branching processes and with the solutions of certain stochastic-fixed point equations arising in the context of the stability of point processes under thinning. We describe the set of solutions to these equations and thus provide a characterization of one-dimensional point processes that are stable with respect to thinning by integer-valued random walks.
}

\bigskip



\section{Introduction}\label{sec:intro}

The classical leader-election procedure \cite{BrussGrubel:03,Fill+Mahmoud+Szpankowski:96,GruebelHagemann:16,Janson+Szpankowski:97,Prodinger:93}, when applied to the infinite set of positive integers $\N$, may be viewed as a random sieve which in each round eliminates an integer not yet sieved in accordance with the outcome of a coin tossing event. The integers are typically viewed as players in a game who independently toss a coin so as to determine whether they will stay in the game for the next round or not. An alternative description
is the following: Relabel kept integers (players) at the beginning of each round by $1,2,\ldots$ while keeping the original order. Then let $R=\{R(k):k\ge 1\}$ be the random set of integers kept for the next round, where
$$ R(0):=0,\quad R(k):=\xi_{1}+\ldots+\xi_{k},\ k\ge 1 $$
is a random walk with independent identically distributed (iid) increments $\xi_{1},\xi_{2},\ldots$ having a geometric distribution on $\N$. Adopting this viewpoint, a natural generalization is to replace the geometric distribution by an arbitrary distribution $(p_{n})_{n\ge 1}$ on $\N$ with $p_{1}<1$.

The purpose of this paper is to study the asymptotics of some relevant quantities for this generalization which will lead us to some interesting connection with Galton-Watson branching processes (GWP) and the solutions of certain related stochastic-fixed point equations (SFPE). Such SFPE's in turn arise in connection with the stability of point processes as will be explained in Section \ref{subsec:stability of PP}.

\begin{figure}[t]
\centering
\includegraphics[width=0.99\textwidth]{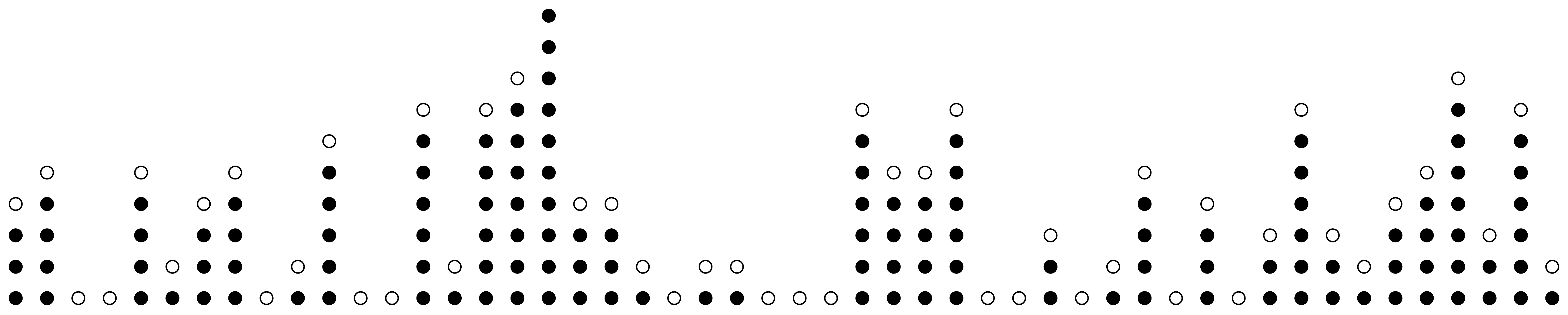}
\caption{A realization of the leader election procedure. The players are arranged horizontally, the vertical axis corresponds to the number of rounds, the bottom line being the first round. Players leaving (staying in) the game are shown as empty (filled) circles.}
	\label{fig1}
\end{figure}

\vspace{.1cm}
A more formal model description is next; see Figure~\ref{fig1} for a sample realization.  Let $R^{(n)}$, $n\ge 1$, be independent copies of a random walk $R$ on $\N$ and denote the increments of $R^{(n)}$ by $\xi_{1}^{(n)}, \xi_{2}^{(n)},\ldots$ That is, the $\xi_{k}^{(n)}$ for $k,n\in\N$ are iid random variables with $\Prob\{\xi_{k}^{(n)}=i\}= p_{i}$, $i\in\N$, and
$$
R^{(n)}(0)=0, \quad R^{(n)}(k)=\xi_1^{(n)} + \ldots + \xi_{k}^{(n)}.
$$
In round $n$, players with current labels $ R^{(n)}(1),R^{(n)}(2),\ldots$ stay for the next round while all other players leave the game. Remaining players are relabeled by $1,2,\ldots$ and the procedure is repeated over and over again. The quantities to be studied hereafter are
\begin{itemize}\itemsep3pt
\item $N^{(n)}_{M}$, the number of players among $1,2,\ldots,M$ who survived the first $n$ rounds, formally
\begin{equation}\label{N_definition}
N_{M}^{(0)}:=M\quad\text{and}\quad N_{M}^{(n)}:=\#\{j\in\N:R^{(n)}(j)\le N_{M}^{(n-1)}\}
\end{equation}
for $M\in\N_{0}:=\{0,1,2,\ldots\}$ and $n\in\N$.
\item $1\le S_{1}^{(n)}<S_{2}^{(n)}<S_{3}^{(n)} <\ldots $, the original numbers of the players who survived the first $n$ rounds, formally
\begin{equation}\label{S_definition}
S_{j}^{(n)}:=\inf\{i\in\N: N_{i}^{(n)}=j\}
\end{equation}
for $j\in\N$ and $n\in\N_0$.
\item $T(M)$, the number of rounds until all players $1,2,\ldots,M$ have been eliminated, thus
\begin{equation}\label{T_definition}
T(M)\ :=\ \inf\{n\in\N:N_{M}^{(n)}=0\}
\end{equation}
for $M\in\N$.
\end{itemize}
The connection with simple GWP's stems from the basic observation that, for each $n\in\N_{0}$,
$$
\Big(S_{1}^{(n)},S_{2}^{(n)},\ldots\Big)\ =\ \Big(S_{R^{(n)}(1)}^{(n-1)},S_{R^{(n)}(2)}^{(n-1)},\ldots\Big),
$$
whence, using the initial conditions $S_{j}^{(0)}=j$ for $j\in\N$,
$$
\Big(S_{1}^{(n)},S_{2}^{(n)},\ldots\Big)\ =\ \Big(R^{(1)}\circ\cdots\circ R^{(n)}(1),R^{(1)}\circ\cdots\circ R^{(n)}(2),\ldots\Big)
$$
for $n\in\N_{0}$, where $\circ$ denotes the usual composition $f\circ g(\cdot)=f(g(\cdot))$. This shows that the random vector $(S_{1}^{(n)},S_{2}^{(n)},\ldots)$ is the $n$-fold forward iteration of the random walk $R$ when viewed as a random mapping from $\N$ to $\N$. Passing to the backward iterations, which does not change the distribution, we obtain the basic relation for our leader-election procedure:
\begin{equation}\label{eq:basic_identity}
\Big(S_{1}^{(n)},S_{2}^{(n)},\ldots\Big)\ \eqdist\ \Big(R^{(n)}\circ\cdots\circ R^{(1)}(1),R^{(n)}\circ\cdots\circ R^{(1)}(2),\ldots\Big)
\end{equation}
for $n\in\N_{0}$. Now, the $j$-th coordinate of the random vector on the right-hand side is nothing but the number of descendants of individuals $1,\ldots,j$ in generation $n$ of a GWP starting from countably many individuals $1,2,\ldots$ and having offspring distribution $(p_{n})_{n\ge 1}$. A sample realization of this GWP is shown in Figure~\ref{fig2}. Since we assume $p_{1}<1$, the GWP is supercritical and survives with probability one.
Two classes of leader-election procedures will be investigated separately hereafter and lead to quite different asymptotics: those generated by a law $(p_{n})_{n\ge 1}$ with finite-mean, in which case the corresponding GWP has also finite mean, and those where $\sum_{n\ge 1}np_{n}=\infty$.

\begin{figure}[t]
\includegraphics[width=0.99\textwidth]{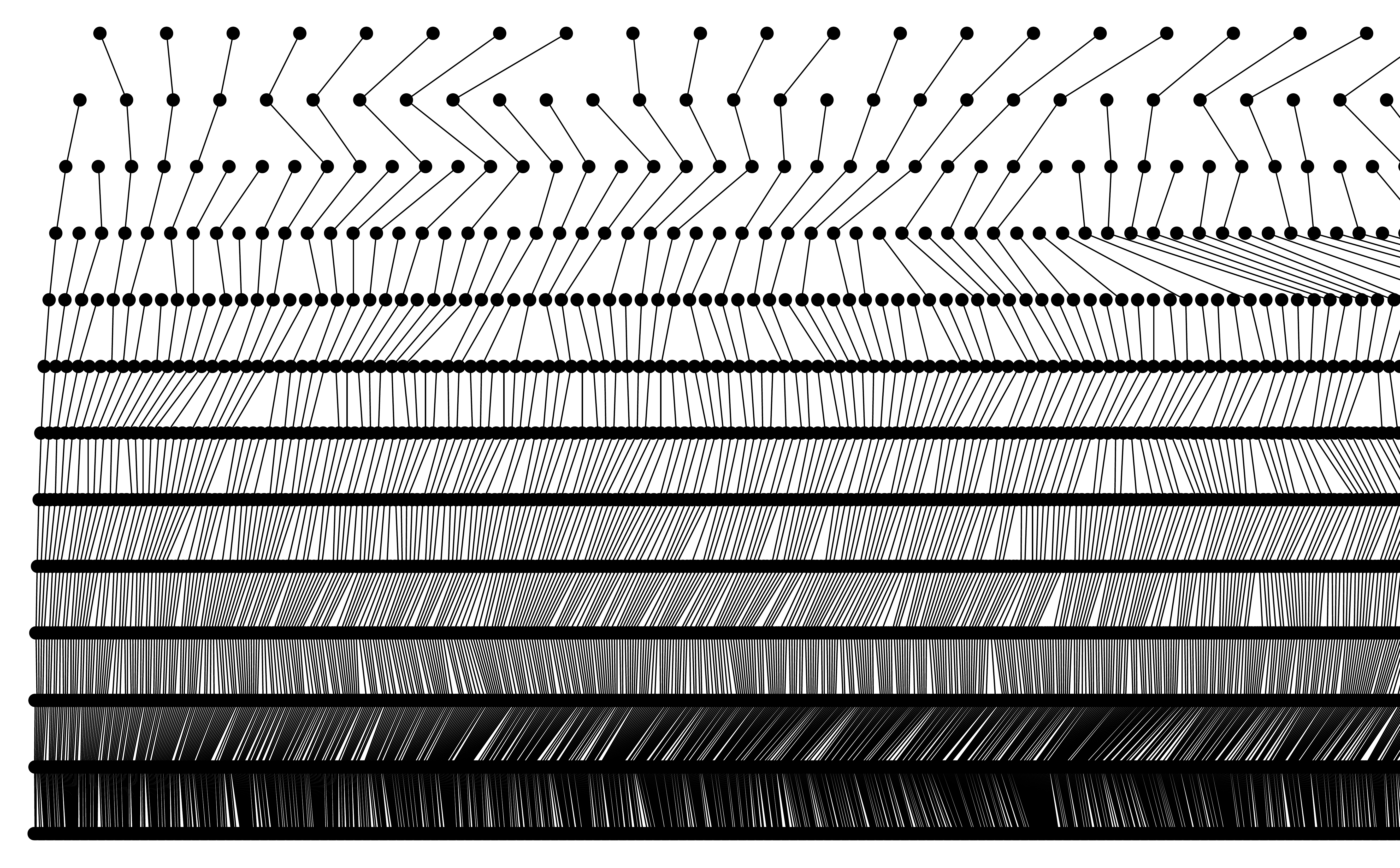}
\caption{A Galton-Watson process starting with countably many individuals $1,2,\ldots$ The $(n+1)$-st row (counting from the top) shows the individuals in the $n$-th generation. The $j$-th individual in generation $n$ is shown at horizontal position $\mu^{-n}j$, where $\mu\in (1,\infty)$ is the expected number of offspring of one individual. Any individual in the $n$-th generation is connected by a line segment to its last descendant in generation $n+1$. For example, individual $1$ in the top row has $2$ direct descendants, whereas individuals $2,\ldots,8$ in the top row have one direct descendant each.}
	\label{fig2}
\end{figure}

\vspace{.1cm}
As indicated by \eqref{eq:basic_identity}, our limit results for $(S_{1}^{(n)},S_{2}^{(n)},\ldots)$, stated as Theorems \ref{main1_for_RW} and \ref{main1_for_RW_heavy}, will be derived from appropriate limit results for GWP's. The asymptotic behavior of $N^{(n)}_{M}$, as $n,M\to\infty$, and $T(M)$, as $M\to\infty$, is then found in a straightforward manner by drawing on the simple duality relations
$$
\{N^{(n)}_M\ge k\}=\{S_{k}^{(n)}\le M\}\quad\text{and}\quad\{T(M)\le k\}=\{S_{1}^{(k)}>M\}
$$
for $k,M\in\N$ and $n\in\N_{0}$, see Theorems \ref{main1_for_RW_N}, \ref{main1_for_RW_T}, \ref{main1_for_RW_heavy_N} and \ref{main1_for_RW_heavy_T}. The limit processes appearing in Theorems \ref{main1_for_RW} and \ref{main1_for_RW_heavy} are solutions to certain stochastic fixed-point equations, see~\eqref{fixed_point_finite_mean} and \eqref{fixed_point_inf_mean}. The description of the set of all solutions to these equations, to be given in Theorems \ref{fixed_point_solutions_fin_mean} and \ref{fixed_point_solutions_inf_mean}, is a much more delicate question, for the necessary analysis heavily relies on deep results about the behavior of GWP.

\vspace{.1cm}
We have organized this work as follows. All results are stated in the next section. Proofs are then provided in Section \ref{sec:finite mean} for the case when $\sum_{n\ge 1}np_{n}<\infty$ and in Section \ref{sec:infinite mean} for the case when $\sum_{n\ge 1}np_{n}=\infty$. Some technical lemmata may be found in a short appendix. Let us finally mention that in \cite{AlsKabMar:15}, we have studied another modification of the classical leader-election procedure which is based on records in an iid sample. This modification is closely related to the Poisson-Dirichlet coalescent and its asymptotic behavior is different from the model studied here.

\section{Results}\label{sec:results}

In the following, $\xi$ always denotes a generic copy of the increments of the random walks $R^{(n)}$ underlying the considered leader-election procedure, thus $\Prob\{\xi=n\}=p_{n}$ for $n\in\N$. Moreover, $\todistrfd$ shall denote convergence of finite dimensional distributions.

\subsection{The case when $\Erw\xi<\infty$}\label{subsec:finite mean}

We start with a convergence result for $(S_{j}^{(n)})_{j\ge 1}$. Put $\mu:=\Erw\xi$ and let $(Z_{n})_{n\ge 0}$ denote a GWP with offspring distribution $(p_{n})_{n\in\N}$. If $\mu$ is finite, thus $\mu\in (1,\infty)$, then the normalization $(\mu^{-n}Z_{n})_{n\ge 0}$ constitutes a positive martingale and converges a.s. to a limit $Z_{\infty}$ which is either a.s. positive, namely if $\Erw\xi\log\xi<\infty$, or zero.

\begin{Theorem}\label{main1_for_RW}
Suppose that $\mu\in (1,\infty)$. Then, as $n\to\infty$,
\begin{equation}\label{eq:RW_S_conv}
\left(\frac{S^{(n)}_{1}}{\mu^{n}},\frac{S^{(n)}_{2}}{\mu^{n}},\frac{S^{(n)}_3}{\mu^{n}}\ldots\right)
\ \todistrfd\
(Z^{(1)}_{\infty}, Z^{(1)}_{\infty}+Z^{(2)}_{\infty}, Z^{(1)}_{\infty}+Z^{(2)}_{\infty}+Z^{(3)}_{\infty}, \ldots),
\end{equation}
where $Z_{\infty}^{(1)},Z_{\infty}^{(2)},\ldots$ are independent copies of $Z_{\infty}$. The distribution of the limit vector in \eqref{eq:RW_S_conv} satisfies the SFPE
\begin{equation}\label{fixed_point_finite_mean}
\Big(\mu X_{1},\mu X_{2},\ldots\Big)\ \eqdist\ \Big(X_{R(1)},X_{R(2)},\ldots\Big),
\end{equation}
where the random walk $(R(j))_{j\ge 0}$ on the right-hand side with generic increment $\xi$ is independent of $(X_{j})_{j\ge 0}$.
\end{Theorem}

In what follows we denote by $D:=D[0,\infty)$ the Skorokhod space of real-valued functions that are defined and right-continuous on $[0,\infty)$, and with finite limits from the left on $(0,\infty)$. Weak convergence on the space $D$ (which may be endowed with the $J_1$- or $M_1$-topology depending on the situation) is denoted by $\todistrD$.

\begin{Theorem}\label{main1_for_RW_N}
Let $\mu\in (1,\infty)$ and $\Erw\xi\log\xi<\infty$. Then, as $n\to\infty$,
\begin{equation*}
N^{(n)}_{\lfloor\mu^{n} \cdot\rfloor}\ \todistrD\ N'(\cdot)
\end{equation*}
weakly in the Skorokhod space $D$ endowed with the $J_1$-topology, where $N'(x):=\#\{k\in\N:$  $Z^{(1)}_{\infty}+\ldots+Z^{(k)}_{\infty}\le x\}$ for $x\ge 0$.
\end{Theorem}

Note that $N'(\cdot)$ is the renewal counting process associated with $(\sum_{j=1}^{k}Z_{\infty}^{(j)})_{k\ge 1}$ and therefore a homogeneous Poisson process if the law of $Z_{\infty}$ is exponential, see Example \ref{exa:classical case} below.

\vspace{.1cm}
The next theorem provides a one-dimensional result for the number $T(M)$ of rounds until all players $1,\ldots,M$ have been eliminated. Although not difficult to prove, we have refrained from a statement of a corresponding functional limit theorem like Theorem \ref{main1_for_RW_N} because it would have required the introduction of a lot more additional notation. See \cite{AlsKabMar:15} for results of this type in a similar setup.

\begin{Theorem}\label{main1_for_RW_T}
Let $\mu\in (1,\infty)$ and $\Erw\xi\log\xi<\infty$. For fixed $x>0$, we have
\begin{equation*}
T(\lfloor\mu^{n} x\rfloor) - n\ \idist\ T'(x),
\end{equation*}
where the distribution of $T'(x)$ is given by
\begin{equation}\label{eq:T_prime_distr}
\Prob\{T'(x)\le k\}\ =\ \Prob\{Z_{\infty}>\mu^{-k}x\},\quad k\in\Z.
\end{equation}
\end{Theorem}

\begin{Rem}\label{rem:SH_limit}\rm
In Theorems \ref{main1_for_RW_N} and \ref{main1_for_RW_T}, we assume $\Erw\xi\log\xi<\infty$ since otherwise $Z_{\infty}=0$ a.s. which means that $N'(x)$ and $T'(x)$ are undefined. If $\mu\in(1,\infty)$ and $\Erw\xi\log\xi=\infty$	it is possible to obtain counterparts of Theorems \ref{main1_for_RW}, \ref{main1_for_RW_N} and \ref{main1_for_RW_T} using a Seneta-Heyde normalization for the GWP $(Z_{n})_{n\ge 0}$. More precisely, if $(c_{n})_{n\in\N_0}$ is such that, as $n\to\infty$,
$$ \frac{Z_{n}}{c_{n}}\ \to \ Z^{(SH)}_{\infty}\quad\text{a.s.} $$ 
where $Z^{(SH)}_{\infty}$ is a.s. positive, then
all the claims of Theorems \ref{main1_for_RW}, \ref{main1_for_RW_N} and \ref{main1_for_RW_T} remain valid upon dropping the assumption $\Erw\xi\log\xi<\infty$ and after replacing $\mu^n$ by $c_{n}$ and $Z_{\infty}$ by $Z^{(SH)}_{\infty}$ everywhere, except the term $\mu^{-k}$ in formula \eqref{eq:T_prime_distr}. The latter formula in that case takes the form
$$
\Prob\{T'(x)\le k\}\ =\ \Prob\{Z^{(SH)}_{\infty}>\mu^{-k}x\},\quad k\in\Z.
$$
\end{Rem}

\begin{Exa}\label{exa:classical case}\rm
In the classical leader-election procedure, the players who stay in the game for the next round are determined by iid Bernoulli trials, so that $\xi$ has a geometric distribution on 
$\N$ with some parameter $p\in (0,1)$. Looking at the number of players who survive the first $n$ rounds then leads to a Bernoulli process with parameter $p^{n}$. It follows that the process on the right-hand side of \eqref{eq:RW_S_conv} is a random walk with standard exponential increments $Z^{(k)}_{\infty}$, $k\in\N$, and $(N'(x))_{x\ge 0}$ a standard Poisson process.
\end{Exa}

An interesting and intriguing problem arising from Theorem \ref{main1_for_RW} is to describe the set of all positive nondecreasing solutions to the SFPE \eqref{fixed_point_finite_mean}. Indeed, we can view \eqref{fixed_point_finite_mean} as a definition of a certain stability property of point processes; see Section \ref{subsec:stability of PP} for more details. Note that if the distribution of a random sequence $(X_{1},X_{2},\ldots)$ satisfies \eqref{fixed_point_finite_mean} and $G:\R^{+}\to\R^{+}$, where $\R^{+}:=[0,\infty)$, is an arbitrary nondecreasing random process independent of $(X_{1},X_{2},\ldots)$ and with the restricted self-similarity property
\begin{equation}\label{G_selfsimilar}
(G(\mu t))_{t\in\R^{+}}\ \ofdd\ (\mu G(t))_{t\in\R^{+}},
\end{equation}
then $(G(X_{1}),G(X_{2}),\ldots)$ also satisfies \eqref{fixed_point_finite_mean}. Property \eqref{G_selfsimilar} is known in the literature under the name \emph{semi-selfsimilarity}. We refer to \cite{MaejimaSato:99} for a general definition of semi-selfsimilar processes and their basic properties. We also mention in passing that the class of semi-selfsimilar processes forms a semigroup with respect to composition of independent realizations.

\vspace{.1cm}
Our next result shows that all solutions to \eqref{fixed_point_finite_mean} can be constructed in the above way and the limit in \eqref{eq:RW_S_conv} provides a solution which is fundamental in a certain sense.

\begin{Theorem}\label{fixed_point_solutions_fin_mean}
Let $(X_{1},X_{2},\ldots)$ be a random element of $\R^{\N}$ such that $0\le X_{1}\le X_{2}\le \ldots$ and \eqref{fixed_point_finite_mean} holds with $(R(j))_{j\ge 0}$ independent of $(X_{j})_{j\in\N}$ and $\mu=\Erw R(1)=\Erw\xi$. Assume further that
\begin{equation}\label{eq:x_log_x}
\Erw R(1)\log R(1)\ =\ \Erw\xi\log\xi\ <\ \infty,
\end{equation}
so that $\Prob(Z_{\infty}>0)=1$. Let $(Z_{\infty}^{(j)})_{j\in\N}$ be independent copies of $Z_{\infty}$. Then there exists a nondecreasing random process $(G(t))_{t\in\R^{+}}$, independent of $(Z_{\infty}^{(j)})_{j\ge 1}$ and satisfying \eqref{G_selfsimilar}, such that
$$
(X_{j})_{j\ge 1}\ \eqdist\ \big(G(Z_{\infty}^{(1)}+\ldots+Z_{\infty}^{(j)})\big)_{j\ge 1}.
$$
\end{Theorem}

\begin{Rem}\label{rem:SH_fixed_point}\rm
Our proof of Theorem \ref{fixed_point_solutions_fin_mean} does not work in the case when $\mu\in(1,\infty)$ and $\Erw\xi\log\xi=\infty$, and we do not know whether the result still holds without the $(\xi\log\xi)$-assumption after replacing $Z_{\infty}$ by $Z^{(SH)}_{\infty}$ as in Remark \ref{rem:SH_limit}. Of course, the direct part of the claim, namely that all vectors of the form $$\big(G(Z^{(SH)(1)}_{\infty}+\ldots+Z^{(SH)(j)}_{\infty})\big)_{j\ge 1}$$ are solutions to \eqref{fixed_point_finite_mean}, is obvious. However, it remains open whether the converse is true.
\end{Rem}

\subsection{The case when $\Erw\xi=\infty$}\label{subsec:infinite mean}

The behavior of GWP's with infinite mean has been studied by various authors including Darling \cite{Darling:70}, Seneta \cite{Seneta:73}, Davies \cite{Davies:78}, Grey \cite{Grey:77}, Schuh and Barbour \cite{BarbourSchuh:77}. In the last reference, it has been proved that for an arbitrary infinite-mean GWP there always exists a function $U$ and a sequence of deterministic constants $(C_{n})_{n\in\N}$ such that $U(Z_{n})/C_{n}$ converges almost surely to a non-degenerate random variable. However, the implicit construction of $U$ and $C_{n}$ makes it difficult to deduce any quantitative result for our model in this general framework. Here we work under the assumption of Davies~\cite{Davies:78}, which to the best of our knowledge, are the most general conditions allowing the explicit construction of $U$ and $C_{n}$. They also guarantee that the a.s.\ limit is positive on the set of survival. Davies' assumption is:
\begin{equation}\label{eq:Davies1}
x^{-\alpha - \gamma(x)}\ \le\ \Prob\{\xi \ge x\}\ \le\ x^{-\alpha + \gamma(x)}, \quad x\ge x_{0},
\end{equation}
for some $0<\alpha<1$, $x_{0}\ge 0$, and a nonincreasing, non-negative function $\gamma(x)$ such that $x^{\gamma(x)}$ is nondecreasing and $\int_{x_{0}}^{\infty} \gamma(\exp(e^{x}))\,dx<\infty$.
For example, this assumption is satisfied if $x^{\alpha}\Prob\{\xi >x\}$ stays bounded away from $0$ and $+\infty$ for $x\ge x_{0}$ (to see this, take $\gamma(x) = C/\log x$). For a GWP $(Z_{n})_{n\ge 0}$ with generic offspring variable $\xi$ satisfying the above assumptions and $Z_{0}=1$, Davies \cite[Thms. 1 and 2]{Davies:78} proved the existence of the limit
\begin{equation}\label{eq:Z_infty1}
Z_{\infty}^{*}\ :=\ \lim_{n\to\infty} \alpha^{n} \log (1+Z_{n})\ \in\ (0,\infty) \quad \text{a.s.}
\end{equation}
Since $\Prob\{\xi=0\}=0$ in our setting, we have $1\le Z_{n}\to\infty$ a.s. and therefore the equivalence of \eqref{eq:Z_infty1} with
\begin{equation}\label{eq:Z_infty}
Z_{\infty}^{*}\ =\ \lim_{n\to\infty} \alpha^{n} \log Z_{n} \in (0,\infty) \quad \text{a.s.}
\end{equation}
Moreover, $Z_{\infty}^{*}$ has a continuous distribution on $(0,\infty)$, see \cite[p.~715]{Hofstad+Hooghiemstra+Znamenski:07} or \cite[bottom of p.~3763]{Athreya:12}.

\begin{Theorem}\label{main1_for_RW_heavy}
Consider a leader-election procedure in which the distribution of $\xi$ satisfies \eqref{eq:Davies1} for some $\alpha\in(0,1)$. Then
\begin{equation}\label{eq:RW_S_conv_heavy}
\left(\alpha^{n} \log S^{(n)}_{j}\right)_{j\ge 1}\ \todistrfd\
\left(Z_{\infty}^{(*,1)}, Z_{\infty}^{(*,1)}\vee Z_{\infty}^{(*,2)},Z_{\infty}^{(*,1)}\vee Z_{\infty}^{(*,2)}\vee Z_{\infty}^{(*,3)}, \ldots\right),
\end{equation}
where $Z_{\infty}^{(*,1)},Z_{\infty}^{(*,2)},\ldots$ are independent copies of $Z_{\infty}^{*}$ and $\vee$ denotes the maximum. The distribution of the limit vector in \eqref{eq:RW_S_conv_heavy} satisfies the SFPE
\begin{equation}\label{fixed_point_inf_mean}
(X_{1},X_{2},\ldots)\ \eqdist\ \left(\alpha X_{R(1)},\alpha X_{R(2)},\ldots\right),
\end{equation}
where the random walk $(R(j))_{j\ge 1}$ on the right-hand side is independent of $(X_{j})_{j\ge 1}$.
\end{Theorem}

On the right-hand side of \eqref{eq:RW_S_conv_heavy}, we thus have the running maximum process of iid positive random variables instead of sums as in the finite-mean case. For a study of this process (including, for example, a proof of the Markov property and exact expressions for the transition probabilities), we refer to~\cite[Lectures 14, 15, 17]{Nevzorov:01}. Let us stress that this process has multiple elements which means that the original numbers of players remaining after $n$ rounds tend to build clusters (at least asymptotically on the log-scale). In fact, it follows from the R\'enyi theorem on records \cite[p.~58]{Nevzorov:01} that among the first $k$ elements of this process there are just $\simeq\log k$ distinct ones, as $k\to\infty$.

\begin{Theorem}\label{main1_for_RW_heavy_N}
Under the assumptions of Theorem \ref{main1_for_RW_heavy}, we have
\begin{equation*}
N^{(n)}_{\lfloor\exp(\cdot\alpha^{-n})\rfloor}\ \todistrD\ N''(\cdot),
\end{equation*}
weakly in the Skorokhod space $D$ endowed with the $M_1$-topology, where $N''(x):=\#\{k\in\N: Z_{\infty}^{(*,1)}\vee\ldots\vee Z_{\infty}^{(*,k)}\le x\}$, $x\ge 0$.
\end{Theorem}

\begin{Rem}\rm
The above theorem breaks down if $D$ is endowed with the $J_1$-topology. Indeed, the sample paths of the process $(N^{(n)}_{\lfloor\exp(x\alpha^{-n})\rfloor})_{x\ge 0}$ belong to the set of piecewise constant non-decreasing functions with jumps of size $1$ which is a closed subset of $D$ in the $J_1$-topology. But the process $N''$ has jumps of size at least $2$ with probability $1$ (due to the clustering), hence Theorem \ref{main1_for_RW_heavy_N} cannot hold when using the $J_1$-topology.
\end{Rem}

The next result is the counterpart of Theorem \ref{main1_for_RW_T} in the infinite-mean case.

\begin{Theorem}\label{main1_for_RW_heavy_T}
Under the assumptions of Theorem \ref{main1_for_RW_heavy}, we have, for any fixed $x>0$,
\begin{equation*}
T([e^{\alpha^{-n}x}])-n\ \idist\ T''(x)
\end{equation*}
as $n\to\infty$, where the distribution of $T''(x)$ is given by
$$
\Prob\{T''(x)\le k\}\ =\ \Prob\{Z_{\infty}^{*}>\alpha^{k}x\}, \quad k\in\Z.
$$
\end{Theorem}

An interesting example of an infinite-mean GWP is obtained by choosing the law of $\xi$, i.e. $(p_{n})_{n\ge 1}$, to be a Sibuya distribution with generating function
$$ f_{\alpha}(t) := \Erw t^{\xi}\ =\ 1 - (1 - t)^{\alpha}, \quad 0\le t \le 1, $$
for some parameter $\alpha\in (0,1)$.

\begin{Prop}\label{Prop_Sibuya}
If $\xi$ has a Sibuya distribution with parameter $\alpha\in (0,1)$, then weak convergence of point processes on $\R^{+}$ holds true, viz.
$$
\sum_{j=1}^{\infty} \delta_{\alpha^{n} \log S_{j}^{(n)}}\ \weakly\ \sum_{i=1}^{\infty} G_{i}\,\delta_{P_{i}},
$$
where $P_{1}<P_{2}<\ldots$ are the points of a standard Poisson process on $(0,\infty)$, and, given these points, the random variables $G_{1},G_{2},\ldots$ are conditionally independent with $G_{j}$ having a geometric distribution with parameter $e^{-P_{j}}$.
\end{Prop}

So we have in the Sibuya case that after normalization with $\alpha^{n}\log x$ and for large $n$, the points $S_{j}^{(n)}$ form approximately a standard Poisson process and have geometrically distributed multiplicities (cluster sizes), their parameters being $\approx (S_{j}^{(n)})^{-\alpha^{n}}$.

\vspace{.1cm}
Our last theorem is the infinite-mean counterpart of Theorem \ref{fixed_point_solutions_fin_mean} and provides the description of the set of all solutions to \eqref{fixed_point_inf_mean} under the Davies' assumption \eqref{eq:Davies1} on $\xi$.

\begin{Theorem}\label{fixed_point_solutions_inf_mean}
Let $(X_{1},X_{2},\ldots)$ be a random element of $\R^{\N}$ such that $0\le X_{1}\le X_{2}\le \ldots$ and \eqref{fixed_point_inf_mean} holds with $(R(j))_{j\ge 0}$ independent of $(X_{j})_{j\in\N}$. Suppose further Davies' condition \eqref{eq:Davies1} and let $(Z_{\infty}^{(*,j)})_{j\in\N}$ be independent copies of $Z^{*}_{\infty}$, defined by \eqref{eq:Z_infty}. Then there exists a nondecreasing random process $G:\R^{+}\to\R^{+}$ satisfying
\begin{equation}\label{G_selfsimilar2}
(G(\alpha t)_{t\in \R^+}\ \eqdist\ (\alpha G(t))_{t\in \R^+}
\end{equation}
and independent of $(Z_{\infty}^{(*,j)})_{j\ge 1}$ such that
$$
(X_{j})_{j\in\N}\ \eqdist\ \left(G(Z_{\infty}^{(*,1)}\vee\ldots\vee Z_{\infty}^{(*,j)})\right)_{j\ge 1}.
$$
\end{Theorem}

\subsection{Stability of point processes: a natural connection}\label{subsec:stability of PP}

The description of the set of solutions to the fixed-point equations \eqref{fixed_point_finite_mean} and \eqref{fixed_point_inf_mean} provided by our Theorems \ref{fixed_point_solutions_fin_mean} and \ref{fixed_point_solutions_inf_mean} may be interpreted from a different point of view involving the notion of stability of point processes, see \cite{DavydovMolchanovZuyev:08,DavydovMolchanovZuyev:11} and \cite{Zanella+Zuyev:15}. To define such stability usually requires two operations, namely \emph{thinning} and \emph{rescaling}. Given a point process $\mathcal{X}:=\sum_{k=1}^{\infty}\delta_{X_{k}}$ in $[0,\infty)$ with $0\le X_{1}\le X_{2}\le \ldots$, and an increasing integer-valued random walk $R=(R(k))_{k\ge 1}$, we define the thinning of $\mathcal{X}$ by $R$ as
$$
\mathcal{X}\bullet R\ :=\ \sum_{k=1}^{\infty}\delta_{X_{R(k)}}.
$$
This random operation transforms $\mathcal{X}$ into a ``sparser'' point process $\mathcal{X}\bullet R$ by removing points of $\mathcal{X}$ with indices outside the range of the random walk $R$. In order to compensate such thinning, a second operation is used for rescaling, namely the usual multiplication $a\cdot\mathcal{X}:=\sum_{k=1}^{\infty}\delta_{aX_{k}}$, $a\in(0,1)$. We call a point process $\mathcal{X}$ {\it $a$-stable with respect to thinning by an integer-valued increasing random walk $R$} if
\begin{equation}\label{eq:stability_definition}
\mathcal{X}\ \eqdist\ a\cdot(\mathcal{X}\bullet R).
\end{equation}
Note that $\mathcal{X}$ is $a$-stable if and only if $\mathcal{X}^{\beta}$ is $a^{\beta}$-stable, where $\mathcal{X}^{\beta}:=\sum_{k=1}^{\infty}\delta_{X_{k}^{\beta}}$ and $\beta>0$. This observation implies that, given a random walk $R$, it is enough to study only $a$-stable point processes for some particular choice of $a\in(0,1)$.

\vspace{.1cm}
Adopting this viewpoint, Theorems \ref{fixed_point_solutions_fin_mean} and \ref{fixed_point_solutions_inf_mean} are nothing else but characterizations of $a$-stable point processes with respect to thinning by random walks. Moreover, the particular choice of $a$ ($a=\mu^{-1}$ in Theorem \ref{fixed_point_solutions_fin_mean} and $a=\alpha$ in Theorem \ref{fixed_point_solutions_inf_mean}) does not reduce generality which means that the aforementioned theorems actually provide the description of the set of solutions to \eqref{eq:stability_definition} for arbitrary $a\in (0,1)$.

\section{Proofs in the finite-mean case}\label{sec:finite mean}

\subsection{Auxiliary results about GWP's with finite-mean offspring distribution}
The results of this section are used in the proof of Theorem~\ref{fixed_point_solutions_fin_mean}.
\begin{Lemma}\label{lemma_xlogx}
Let $\theta$ be a positive random variable with $\mu:=\Erw\theta\in (1,\infty)$ and $\Erw\theta\log^{+}\theta<\infty$. Then, for each $p\in [1,2)$,
$$
\sum_{k=1}^{\infty}\Big(\Erw|\theta-\mu|\1_{\{\mu^{k-1}<|\theta-\mu|\le\mu^{k}\}}\Big)^{1/p}<\infty.
$$
\end{Lemma}

\begin{proof}
The statement is obvious for $p=1$ as the series on the left-hand side then reduces to $\Erw|\theta-\mu|\1_{\{|\theta-\mu|>1\}}$. So let $p\in(1,2)$ and choose $q>2>p$ such that $1/p+1/q=1$. Using H\"{o}lder's inequality $|\sum_{k} a_{k}b_{k}|\le (\sum_{k} |a_{k}|^p)^{1/p}(\sum_{k} |b_{k}|^q)^{1/q}$ with $b_{k}=k^{-1/p}$ and $a_{k}=\Big(k\Erw|\theta-\mu|\1_{\{\mu^{k-1}<|\theta-\mu|\le\mu^{k}\}}\Big)^{1/p}$, we obtain
\begin{align*}
\sum_{k=1}^{\infty}&\Big(\Erw|\theta-\mu|\1_{\{\mu^{k-1}<|\theta-\mu|\le\mu^{k}\}}\Big)^{1/p}\\
&\le\ \left(\sum_{k=1}^{\infty}k\,\Erw|\theta-\mu|\1_{\{\mu^{k-1}<|\theta-\mu|\le\mu^{k}\}}\right)^{1/p}\left(\sum_{k=1}^{\infty}k^{-q/p}\right)^{1/q}.
\end{align*}
The second series converges because $q>p$ and the first one can be bounded by
$$
\Erw\left(\sum_{k=1}^{\infty}\Big(\frac{\log|\theta-\mu|}{\log\mu}+1\Big)|\theta-\mu|\1_{\{\mu^{k-1}<|\theta-\mu|\le\mu^{k}\}}\right)
$$
which is finite because $\Erw\theta\log^{+}\theta<\infty$.\qed
\end{proof}

\begin{Lemma}\label{summable_lemma}
Let $(\theta^{(n)}_{k})_{n\in\N,\,k\in\N}$ be an array of independent copies of a positive random variable $\theta$ having $\mu=\Erw\theta\in (1,\infty)$ and $\Erw\theta\log^{+}\theta<\infty$. For $n\in\N$, define the increasing random walks
$$
R^{(n)}(0):=0,\quad R^{(n)}(k):=\theta^{(n)}_{1}+\ldots+\theta^{(n)}_{k},\quad k\in\N.
$$
Then, for any $T>0$,
$$
\sum_{n=1}^{\infty}\sup_{t\in[0,T]}\Big|\mu^{-n}R^{(n)}(\lfloor t\mu^{n-1}\rfloor)-t\Big|\ <\ \infty\quad\text{a.s.}
$$
\end{Lemma}
\begin{proof}
Defining $\theta_{\leqslant,k}^{(n)}:=(\theta_{k}^{(n)}-\mu)\1_{\{|\theta_{k}^{(n)}-\mu|\le \mu^{n}\}}$ and $\theta_{>,k}^{(n)}:=(\theta_{k}^{(n)}-\mu)\1_{\{|\theta_{k}^{(n)}-\mu|>\mu^{n}\}}$, we have
\begin{align*}
\sup_{t\in[0,\,T]}\left|\mu^{-n}R^{(n)}(\lfloor t\mu^{n-1}\rfloor)-t\right|\ &\le\ \mu^{-n}\sup_{t\in[0,\,T]}|R^{(n)}(\lfloor t\mu^{n-1}\rfloor)-\mu\lfloor t\mu^{n-1}\rfloor|\\
&+\ \mu^{-n}\sup_{t\in[0,T]}|\mu\lfloor t\mu^{n-1}\rfloor-t\mu^{n}|\\
&\hspace{-2.5cm}\le\ \mu^{-n}\left(\sup_{t\in[0,T]}\left|\sum_{k=1}^{\lfloor t\mu^{n-1}\rfloor}\theta_{\leqslant,k}^{(n)}\right|+\sup_{t\in[0,T]}\left|\sum_{k=1}^{\lfloor t\mu^{n-1}\rfloor}\theta_{>,k}^{(n)}\right|\right)+\mu^{-n}.
\end{align*}
Hence it is enough to show that
\begin{equation}\label{eq:two_series}
\sum_{n=1}^{\infty}\mu^{-n}\sup_{0\le m\le \lfloor T\mu^{n-1}\rfloor}\left|\sum_{k=1}^{m}\theta_{\leqslant,k}^{(n)}\right|\quad\text{and}\quad\sum_{n=1}^{\infty}\mu^{-n}\sup_{0\le m\le\lfloor T\mu^{n-1}\rfloor}\left|\sum_{k=1}^{m}\theta_{>,k}^{(n)}\right|
\end{equation}
are both almost surely finite. As for the second series, this follows by the Borel-Cantelli lemma if we can show that
\begin{equation}\label{eq:summable_lemma}
\sum_{n=1}^{\infty}\Prob\left\{\text{there exists }k=1,\ldots,\lfloor T\mu^{n-1}\rfloor:\theta_{>,k}^{(n)}\ne 0\right\}\ <\ \infty.
\end{equation}
To this end, we use Boole's inequality to infer
\begin{align*}
\sum_{n=1}^{\infty}&\Prob\left\{\text{there exists }k=1,\ldots,\lfloor T\mu^{n-1}\rfloor:\theta_{>,k}^{(n)}\ne 0\right\}\ \le\ \sum_{n=1}^{\infty}(T\mu^{n-1})\Prob\left\{\theta_{>,1}^{(n)}\ne 0\right\}\\
&=\ \sum_{n=1}^{\infty}(T\mu^{n-1})\,\Prob\{|\theta-\mu|>\mu^{n}\}\ =\ \Erw\left(\sum_{n=1}^{\infty}(T\mu^{n-1})\1_{\{|\theta-\mu|>\mu^{n}\}}\right)\\
&\le\ \frac{T}{\mu}\,\Erw\left(\sum_{n=1}^{\infty}|\theta-\mu|\1_{\{|\theta-\mu|>\mu^{n}\}}\right)\ \le\ \frac{T}{\mu}\,\Erw\left((\theta+\mu)\sum_{n=1}^{\infty}\1_{\{\theta\ge \mu^{n}\}}\right)\\
&\le\ \frac{T}{\mu}\,\Erw\left((\theta+\mu)\sum_{n=1}^{\infty}\1_{\{\log^{+}\theta\ge n\log\mu\}}\right)\ \le\ \frac{T}{\mu}\,\Erw\left((\theta+\mu)\frac{\log^{+}\theta}{\log\mu}\right)\ <\ \infty
\end{align*}
where the finiteness of the last term follows from $\Erw\theta\log^{+}\theta<\infty$. To show that the first series in \eqref{eq:two_series} converges, we argue as follows:
\begin{align*}
\sum_{n=1}^{\infty}\mu^{-n}&\sup_{0\le m\le \lfloor T\mu^{n-1}\rfloor}\left|\sum_{k=1}^{m}\theta_{\leqslant,k}^{(n)}\right|\\
&\le\ \sum_{n=1}^{\infty}\mu^{-n}\sup_{0\le m\le\lfloor T\mu^{n-1}\rfloor}\left|\sum_{k=1}^{m}\theta_{\leqslant,k}^{(n)}-m\Erw\theta_{\leqslant,1}^{(n)}\right|\ +\ \frac{T}{\mu}\sum_{n=1}^{\infty}|\Erw\theta_{\leqslant,1}^{(n)}|.
\end{align*}
The last term on the right-hand side is finite because
$$
\sum_{n=1}^{\infty}|\Erw\theta_{\leqslant,1}^{(n)}|\ =\ \sum_{n=1}^{\infty}|\Erw\theta_{>,1}^{(n)}|\ \le\ \sum_{n=1}^{\infty}\Erw |\theta-\mu|\1_{\{|\theta-\mu|>\mu^{n}\}}\ <\ \infty,
$$
where the finiteness of the last sum has already been shown above. To bound the first term, note that $\big(\sum_{k=1}^{m}\big(\theta_{\leqslant,k}^{(n)}-\Erw\theta_{\leqslant,1}^{(n)}\big)\big)_{m\in\N}$ is an $L^{p}$-martingale for each $p\in(1,2]$ whence
\begin{align*}
&\Erw\left(\sup_{0\le m\le \lfloor T\mu^{n-1}\rfloor}\left|\sum_{k=1}^{m}\big(\theta_{\leqslant,k}^{(n)}-\Erw\theta_{\leqslant,1}^{(n)}\big)\right|\,\right)\ \le\ \left\|\sup_{0\le m\le \lfloor T\mu^{n-1}\rfloor}\left|\sum_{k=1}^{m}\big(\theta_{\leqslant,k}^{(n)}-\Erw\theta_{\leqslant,1}^{(n)}\big)\right|\,\right\|_{p}\\
&\hspace{.5cm}\le\ \frac{p}{p-1}\left\|\sum_{k=1}^{\lfloor T\mu^{n-1}\rfloor}\big(\theta_{\leqslant,k}^{(n)}-\Erw\theta_{\leqslant,1}^{(n)}\big)\right\|_p\ \le\ \frac{2p}{p-1}T^{1/p}\mu^{(n-1)/p}\left\|\theta_{\leqslant,1}^{(n)}-\Erw\theta_{\leqslant,1}^{(n)}\right\|_{p}\\
&\hspace{1cm}\le\ \frac{4p}{p-1}T^{1/p}\mu^{n/p}\left\|\theta_{\leqslant,1}^{(n)}\right\|_{p},\end{align*}
having utilized the inequalities by Doob and von Bahr-Esseen \cite[Formula 4]{BahrEsseen:65}. Put $q:=\frac{p}{p-1}$, thus $1/p+1/q=1$, and use the inequality $|\sum_{i} x_{i}|^{1/p}\le \sum_{i} |x_{i}|^{1/p}$ to infer
\begin{align*}
\sum_{n=1}^{\infty}\mu^{-n/q}\left\|\theta_{\leqslant,1}^{(n)}\right\|_{p}\ &\le\ \sum_{n=1}^{\infty}\mu^{-n/q}\left(1+\sum_{k=1}^{n}\Erw|\theta-\mu|^{p}\1_{\{\mu^{k-1}<|\theta-\mu|\le\mu^{k}\}}\right)^{1/p}\\
&\le\ \sum_{n=1}^{\infty}\mu^{-n/q}\left(1+\sum_{k=1}^{n}\mu^{k/q}\left(\Erw|\theta-\mu|\1_{\{\mu^{k-1}<|\theta-\mu|\le\mu^{k}\}}\right)^{1/p}\right)
\end{align*}
and then further (with $C$ denoting a suitable finite positive constant)
\begin{align*}
\sum_{n=1}^{\infty}\mu^{-n/q}&\sum_{k=1}^{n}\mu^{k/q}\left(\Erw|\theta-\mu|\1_{\{\mu^{k-1}<|\theta-\mu|\le\mu^{k}\}}\right)^{1/p}\\
&=\ \sum_{k=1}^{\infty}\mu^{k/q}\left(\Erw|\theta-\mu|\1_{\{\mu^{k-1}<|\theta-\mu|\le\mu^{k}\}}\right)^{1/p}\sum_{n=k}^{\infty}\mu^{-n/q}\\
&\le\ C\sum_{k=1}^{\infty}\left(\Erw|\theta-\mu|\1_{\{\mu^{k-1}<|\theta-\mu|\le\mu^{k}\}}\right)^{1/p},
\end{align*}
which is finite by Lemma \ref{lemma_xlogx} if $p\in(1,2)$. We thus arrive at
$$
\Erw\sum_{n=1}^{\infty}\mu^{-n}\sup_{0\le m\le \lfloor T\mu^{n-1}\rfloor}\left|\sum_{k=1}^{m}\theta_{\leqslant,k}^{(n)}-m\Erw\theta_{\leqslant,1}^{(n)}\right|\ <\ \infty
$$
and this completes the proof of the lemma.\qed
\end{proof}

Lemma \ref{summable_lemma} allows us to prove the following proposition which is the key ingredient to the proof of Theorem \ref{fixed_point_solutions_fin_mean}.

\begin{Prop}\label{Prop:GW_summary}
Let $(R^{(n)}(k))_{k\in\N_{0},n\in\N}$ be as in Lemma \ref{summable_lemma} with $\theta$ taking positive integer values only and put
\begin{equation}\label{eq:g_def} 
g_{n}(t)\ :=\ \mu^{-n}R^{(n)}(\lfloor\mu^{n-1}t\rfloor) 
\end{equation}
for $n\in\N$ and $t\ge 0$.
\begin{description}[(A.4)]\itemsep2pt
\item[(A1)] There exists a $D$-valued random process $(Z_{\infty}(t))_{t\ge 0}$ such that
$$
\big(g_{n}\circ \cdots\circ g_{1}(t)\big)_{t\ge 0}\ \stackrel{n\to\infty}{\longrightarrow}\ \big(Z_{\infty}(t)\big)_{t\ge 0}\quad\text{a.s.}
$$
in the space $D$ endowed with the $J_1$-topology. The process $(Z_{\infty}(t))_{t\ge 0}$ is the limit of the normalized number of descendants of individuals $1,\ldots, \lfloor t\rfloor$ in a GWP with generic offspring variable $\theta$ and countably many ancestors $1,2,\ldots$, thus
$$
Z_{\infty}(t)=Z_{\infty}^{(1)}+\ldots+Z_{\infty}^{(\lfloor t\rfloor)}, \quad t\ge 0,
$$
where the $Z_{\infty}^{(j)}$, $j\in\N$, denote independent copies of $Z_{\infty}$, the limit of the same normalized GWP with one ancestor.

\item[(A2)] For each $k\in\N_{0}$, there exists a copy $(Z_{k,\infty}(t))_{t\ge 0}$ of $(Z_{\infty}(t))_{t\ge 0}$ such that
$$
\big(g_{n}\circ \cdots\circ g_{k+1}(t)\big)_{t\ge 0}\ \stackrel{n\to\infty}{\longrightarrow}\ \big(\mu^{-k}Z_{k,\infty}(t\mu^k)\big)_{t\ge 0}\quad\text{a.s.}
$$
in the space $D$ endowed with the $J_1$-topology.

\item[(A3)] As $t\to\infty$,
$$
\inf_{k\ge 1}\,g_{k}\circ\cdots\circ g_{1}(t)\ \to\ \infty\quad\text{a.s.}
$$

\item[(A4)] For every fixed $T>0$,
$$
\sup_{k\ge 1}\sup_{t\in[0,T]}\Big|g_{n+k}\circ\cdots\circ g_{n+1}(t)-t\Big|\ \stackrel{n\to\infty}{\longrightarrow}\ 0\quad\text{a.s.}
$$
\item[(A5)] The set $\big\{\mu^{-k}Z_{k,\infty}(t\mu^k):t\ge 0,k\in\N\big\}$ is almost surely dense in $[0,\infty)$.
\end{description}
\end{Prop}

\begin{Rem}\rm
To motivate the definition of $g_{n}$ in \eqref{eq:g_def}, consider a GWP starting from countably many individuals $1,2,\ldots$, as shown in Figure~\ref{fig2}. Since the expected number of offspring of one individual is $\mu$, it is natural to place the $j$-th individual in generation $n$ at position $\mu^{-n} j$. By definition of $g_{n}$, we have $g_{n}(\mu^{-(n-1)} j) = \mu^{-n} R^{(n)}(j)$ for all $j\in\N$, so that $g_{n}$ maps the position of the $j$-th individual in generation $n-1$ to the position of its last offspring in generation $n$. By the law of large numbers, $g_{n}(t)$ should be close to $t$ for large $n$. Part (A4) of Proposition \ref{Prop:GW_summary} provides the confirmation of this in a very strong uniform sense, and it may also be assessed graphically in Figure~\ref{fig2} where slope line segments connecting individuals in consecutive generations become closer and closer to vertical lines as $n$ grows.
\end{Rem}

\begin{proof}[of Proposition \ref{Prop:GW_summary}]
(A1) Since $g_{n}\circ \cdots\circ g_{1}(t) = \mu^{-n}R^{(n)} \circ \cdots \circ R^{(1)}(\lfloor t\rfloor)$, the assertion is equivalent to
\begin{equation}\label{eq:RW_proof_1}
\left(\mu^{-n}R^{(n)} \circ \cdots \circ R^{(1)}(j)\right)_{j\in\N_{0}}\ \stackrel{n\to\infty}{\longrightarrow}\ \left(Z_{\infty}(j)\right)_{j\in\N_{0}}\quad\text{a.s.}
\end{equation}
For each $n$, the sequence on the left-hand side constitutes a random walk (that is, it has iid increments) since a composition of independent increasing, $\N$-valued random walks is again a random walk (a similar statement in the theory of L\'evy processes is well-known and called Bochner's subordination). Hence, in order to prove~\eqref{eq:RW_proof_1}, it suffices to show that
\begin{equation}\label{eq:RW_proof_{2}}
\mu^{-n}R^{(n)} \circ \cdots \circ R^{(1)}(1)\ \stackrel{n\to\infty}{\longrightarrow}\ Z^{(1)}_{\infty}\quad\text{a.s.}
\end{equation}
The quantity on the left is the normalized number of individuals at time $n$ in a GWP with generic offspring variable $\theta$ and one ancestor, and \eqref{eq:RW_proof_{2}} follows from the a.s.\ convergence of its normalization (which is a nonnegative martingale).

\vspace{0.2cm}
(A2) This follows from (A1), when using the representation
$$
g_{n}\circ \cdots\circ g_{k+1}(t)\ =\ \mu^{-k}\mu^{-(n-k)}R^{(n)} \circ \cdots \circ R^{(k+1)}(\lfloor t\mu^k\rfloor),\quad t\ge 0,
$$
valid for $n>k$.

\vspace{0.2cm}
(A3) Here we have
$$
g_{k}\circ\cdots\circ g_{1}(t)\ =\ \sum_{j=1}^{\lfloor t\rfloor}\mu^{-k}Z^{(j)}_{k},\quad t\ge 0,
$$
where the $(Z^{(j)}_{k})_{k\ge 0}$, $j\in\N$, are independent copies of a GWP $(Z_{k})_{k\ge 0}$ with generic offspring variable $\theta$ and $Z_{0}=1$. Passing to the infimum yields
$$
\inf_{k\in\N}g_{k}\circ\cdots\circ g_{1}(t)\ \ge\ \sum_{j=1}^{\lfloor t\rfloor}\inf_{k\ge 1}(\mu^{-k}Z^{(j)}_{k}),\quad t\ge 0
$$
and the result follows from $\Prob\big\{\inf_{k\ge 1}(\mu^{-k}Z^{(1)}_{k})=0\big\}=0$.

\vspace{.2cm}
(A4) Fix $T>0$. From (A3), we know that there exists a random $T_{1}>0$ such that
$$
\inf_{k\ge 1}g_{k}\circ\cdots\circ g_{1}(T_{1})\ \ge\ T\quad\text{a.s.}
$$
We further note that
\begin{align*}
A_{1}\ &:=\ \sup_{k\ge 0}\,\sup_{n>k}\,\sup_{t\in[0,\,T]}\,g_{n}\circ\cdots\circ g_{k+1}(t)\ =\ \sup_{k\ge 0}\,\sup_{n>k}\,g_{n}\circ\cdots\circ g_{k+1}(T)\\
&\,\le\ \sup_{k\ge 0}\,\sup_{n>k}\,g_{n}\circ\cdots\circ g_{1}(T_{1})\ =\ \sup_{n\ge 1}\,g_{n}\circ\cdots\circ g_{1}(T_{1})\ <\ \infty\quad\text{a.s.},
\end{align*}
the finiteness being ensured by (A1), and that
$$
\big|g_{n+k}\circ\cdots\circ g_{n+1}(t)-t\big|\ \le\ \big|g_{n+1}(t)-t\big|\,+\,\sum_{j=2}^{k}\big|g_{n+j}\circ\cdots\circ g_{n+1}(t)-g_{n+j-1}\circ\cdots\circ g_{n+1}(t)\big|
$$
by the triangle inequality. Consequently,
\begin{align*}
\sup_{k\in\N}&\sup_{t\in[0,\,T]}\big|g_{n+k}\circ\cdots\circ g_{n+1}(t)-t\big|\\
&\le\ \sup_{t\in[0,\,T]}\big|g_{n+1}(t)-t\big|\ +\ \sum_{j=2}^{\infty}\sup_{t\in[0,\,T]}\big|g_{n+j}\circ\cdots\circ g_{n+1}(t)-g_{n+j-1}\circ\cdots\circ g_{n+1}(t)\big|\\
&\le\ \sup_{t\in[0,\,T]}\big|g_{n+1}(t)-t\big|\ +\ \sum_{j=2}^{\infty}\sup_{s\in[0,\,A_{1}]}\big|g_{n+j}(s)-s\big|\ \stackrel{n\to\infty}{\longrightarrow}\ 0
\end{align*}
by Lemma \ref{summable_lemma}.

\vspace{0.2cm}
(A5) By (A4), for any $s\in [0,\infty)$ and $\eps>0$, there exists a random $m\in\N$ such that
\begin{equation}\label{eq:GW_dense_{1}}
s-\eps/2\ \le\ g_{m+k}\circ\cdots\circ g_{m+1}(s)\ \le\ s+\eps/2
\end{equation}
for all $k\in\N$. Choosing $k$ sufficiently large and applying (A2), we obtain
\begin{equation}\label{eq:GW_dense_{2}}
\mu^{-m}Z_{m,\infty}(s\mu^m)-\eps/2\ \le\ g_{m+k}\circ\cdots\circ g_{m+1}(s)\ \le\ \mu^{-m}Z_{m,\infty}(s\mu^m)+\eps/2.
\end{equation}
Finally, \eqref{eq:GW_dense_{1}} and \eqref{eq:GW_dense_{2}} imply
$$
s-\eps\ \le\ \mu^{-m}Z_{m,\infty}( s\mu^m)\ \le\ s+\eps
$$
and the proof is complete.\qed
\end{proof}

\subsection{Proof of Theorems \ref{main1_for_RW}, \ref{main1_for_RW_N}, \ref{main1_for_RW_T} and \ref{fixed_point_solutions_fin_mean}}


Define a random map $\psi:\R^{\N}\to\R^{\N}$ by
\begin{equation}\label{eq:psi_mapping_definition}
\psi((x_{1},x_{2},\ldots))\ :=\ \frac{1}{\mu}\left(x_{R(1)},x_{R(2)},\ldots\right)
\end{equation}
and further $\psi_{n}:\R^{\N}\to\R^{\N}$ for $n\in\N$ by
\begin{equation}\label{eq:psi_seq_mapping_definition}
\psi_{n}((x_{1},x_{2},\ldots))\ :=\ \frac{1}{\mu}\left(x_{R^{(n)}(1)},x_{R^{(n)}(2)},\ldots\right)
\end{equation}
which are independent copies of $\psi$.

\begin{proof}[of Theorem \ref{main1_for_RW}]
In view of the basic identity \eqref{eq:basic_identity}, we have
\begin{align*}
\Big(\frac{S_{1}^{(n)}}{\mu^{n}},\frac{S_{2}^{(n)}}{\mu^{n}},\ldots\Big)\ &\eqdist\
\Big(\mu^{-n}R^{(n)}\circ\cdots\circ R^{(1)}(1),\mu^{-n}R^{(n)}\circ\cdots\circ R^{(1)}(2),\ldots\Big)
\\
&=\ \psi^{(1)}\circ\cdots\circ \psi^{(n)}(1,2,\ldots).
\end{align*}
By (A1) of Proposition \ref{Prop:GW_summary}, the last sequence converges a.s. to $(Z_{\infty}^{(1)},Z_{\infty}^{(1)}+Z_{\infty}^{(2)},\ldots)$ which proves \eqref{eq:RW_S_conv}. The fixed-point relation \eqref{fixed_point_finite_mean} follows from the almost sure continuity of the map $\psi$
with respect to the product topology on $\R^{\N}$ and the continuous mapping theorem. This completes the proof of Theorem \ref{main1_for_RW}.\qed
\end{proof}

\begin{proof}[of Theorem \ref{main1_for_RW_N}]
Observe  that $(N^{(n)}_{\lfloor x\mu^{n}\rfloor})_{x\ge 0}$ is the counting process associated with the increasing random sequence $(\mu^{-n} S_{k}^{(n)})_{k\in\N}$. By Theorem \ref{main1_for_RW}, the finite-dimensional distributions of the latter sequence converge as $n\to\infty$ to those of the strictly increasing random sequence  $(Z_{\infty}^{(1)} + \ldots+ Z_\infty^{(k)})_{k\in\N}$ with associated counting process $(N'(x))_{x\ge 0}$. Also, for every fixed $n\in\N$,
$$
\lim_{k\to\infty}  S_{k}^{(n)}\ =\ \lim_{k\to\infty} (Z_{\infty}^{(1)} + \ldots+ Z_\infty^{(k)})\ =\ + \infty \quad \text{a.s.}
$$
because $\xi>0$ a.s.\  and $Z_\infty^{(1)}>0$ a.s. By Lemma \ref{lem:skorokhod_counting} from the Appendix, this implies the weak convergence of the corresponding counting processes on the Skorokhod space $D$ endowed with the $J_1$-topology.
\end{proof}


\begin{proof}[of Theorem \ref{main1_for_RW_T}]
For this result, it suffices to note
\begin{align*}
\Prob\left\{T(\lfloor\mu^{n} x\rfloor)-n\le k\right\}\ &=\ \Prob\left\{S_{1}^{(n+k)}>\lfloor\mu^{n}x\rfloor\right\}\\
&=\Prob\left\{S_{1}^{(n+k)}>\mu^{n} x\right\}\ =\ \Prob\left\{\mu^{-(n+k)}S_{1}^{(n+k)}>\mu^{-k} x\right\}.\quad\qed
\end{align*}
\end{proof}

\begin{proof}[of Theorem \ref{fixed_point_solutions_fin_mean}]
Let $(X^{(0)}_{1},X^{(0)}_{2},\ldots)$ be a solution to \eqref{fixed_point_finite_mean}, i.e.
$$
\big(\mu X^{(0)}_{1},\mu X^{(0)}_{2},\ldots\big)\ \eqdist\ \left(X^{(0)}_{R(1)},X^{(0)}_{R(2)},\ldots\right).
$$
By the Kolmogorov consistency theorem, the underlying probability space $(\Omega,\fA,\Prob)$ may be assumed to be large enough to carry the following objects:
\begin{itemize}\itemsep2pt
\item the random sequence $(X^{(0)}_{1},X^{(0)}_{2},\ldots)$;
\item a two-sided sequence $(R^{(n)}(\cdot))_{n\in\Z}$ of independent copies of the random walk $R(\cdot)$ and the corresponding sequence of random maps \eqref{eq:psi_seq_mapping_definition};
\item a two-sided stationary sequence $(V_{k})_{k\in\Z}=\big((X_{1}^{(k)},X_{2}^{(k)},\ldots)\big)_{k\in\Z}$ such that $V_{k}$ is independent of $(\psi_{n})_{n\le k}$ for each $k\in\Z$, and
$$ V_{0}\ :=\ \big(X^{(0)}_{1},X^{(0)}_{2},\ldots\big)\quad\text{and}\quad V_{k}\ =\ \psi_{k+1}(V_{k+1}),\quad k\in\Z, $$
thus
$$ \big(X^{(k)}_{1},X^{(k)}_{2},\ldots\big)\ =\ \frac{1}{\mu}\left(X^{(k+1)}_{R^{(k+1)}(1)},X^{(k+1)}_{R^{(k+1)}(2)},\ldots\right),\quad k\in\Z. $$
\end{itemize}
By construction, $V_{k}\eqdist V_{0}$ for $k\in\Z$. Define a sequence of random measures
$(\nu_{n})_{n\ge 0}$ on $[0,\infty)$ by\,\footnote{For ease of notation, we write $\nu_{n}[0,t]$ instead of $\nu_{n}([0,t])$.}
$$
\nu_{n}[0,t]\ :=\
\begin{cases}
\hfill 0, &\text{if }t<\mu^{-n},\\
\mu^{-n}X^{(n)}_{\lfloor t\mu^{n}\rfloor},&\text{if } t\ge \mu^{-n},
\end{cases}
$$
which is possible because $(X^{(0)}_{j})_{j\ge 1}$ is nondecreasing and $(X^{(n)}_{j})_{j\ge 1}\eqdist (X^{(0)}_{j})_{j\ge 1}$ for each $n\in\N_{0}$. Let us assume for a moment that $\nu_{n}$ converges almost surely, as $n\to\infty$, to some limit random measure $\nu_{\infty}$
in the vague topology on $[0,\infty)$, i.e.
\begin{equation}\label{eq:vague_as_convergence}
\nu_{n}\ \vag\ \nu_{\infty}\quad\text{a.s.}
\end{equation}
Let us show that $(X_{k}^{(0)})_{k\ge 1}\eqdist(G(Z_{\infty}^{(1)}),G(Z_{\infty}^{(1)}+Z_{\infty}^{(2)}),\ldots)$ with $G(t):=\nu_{\infty}[0,t]$. Indeed,
\begin{align*}
\left(X^{(0)}_{1},X^{(0)}_{2},\ldots\right)\ &=\ \mu^{-n}\Big(X^{(n)}_{\lfloor\mu^{n} \wh{Z}_{n}(1)\rfloor},X^{(n)}_{\lfloor\mu^{n} \wh{Z}_{n}(2)\rfloor},\ldots\Big)\\
&=\Big(\nu_{n}[0,\wh{Z}_{n}(1)],\nu_{n}[0,\wh{Z}_{n}(2)],\ldots\Big),\quad n\in\N,
\end{align*}
where $\wh{Z}_{n}(j):=\mu^{-n}(R^{(n)}\circ\cdots\circ R^{(1)})(j)$ is independent of $\nu_{n}$. As already pointed out,
\begin{equation}\label{eq:GW_conv1}
\Big(\wh{Z}_{n}(j)\Big)_{j\ge 1}\ \stackrel{n\to\infty}{\longrightarrow}\ \Big(Z_{\infty}^{(1)}+\ldots+Z_{\infty}^{(j)}\Big)_{j\ge 1}\quad\text{a.s.}
\end{equation}
which in combination with \eqref{eq:vague_as_convergence} implies
$$
\Big(\nu_{n},\Big(\wh{Z}_{n}(j)\Big)_{j\in\N}\Big)\ \stackrel{n\to\infty}{\longrightarrow}\ \Big(\nu_{\infty},\Big(Z_{\infty}^{(1)}+\ldots+Z_{\infty}^{(j)}\Big)_{j\in\N}\Big)\quad\text{a.s.}
$$
in the product topology, the components of the limit vector on the right-hand side being independent. Condition \eqref{eq:x_log_x} entails that the law of $Z_{\infty}^{(1)}$ is absolutely continuous, see \cite[Corollary 4 on p.~36]{Athreya+Ney:72}, whence
$$
\Prob\left\{\nu_{\infty}(\{Z_{\infty}^{(1)}+\ldots+Z_{\infty}^{(j)}\})=0\right\}\ =\ 1
$$
for all $j\in\N$. Lemma \ref{lemma_continuity} in the Appendix now yields
\begin{align*}
\big(X^{(0)}_{1},X^{(0)}_{2},\ldots\big)\ =~&\left(\nu_{n}[0,\,\wh{Z}_{n}(1)],\nu_{n}[0,\,\wh{Z}_{n}(2)],\ldots\right)\\
\stackrel{n\to\infty}{\longrightarrow}~&\left(\nu_{\infty}[0,\,Z_{\infty}^{(1)}],\nu_{\infty}[0,\,Z_{\infty}^{(1)}+Z_{\infty}^{(2)}],\ldots\right) \quad \text{a.s.},
\end{align*}
which shows the asserted representation of $\big(X^{(0)}_{1},X^{(0)}_{2},\ldots\big)$ as a solution to \eqref{fixed_point_finite_mean}.

\vspace{.1cm}
It remains to prove \eqref{eq:vague_as_convergence}. Recall that $g_{n}(t):=\mu^{-n}R^{(n)}(\lfloor t\mu^{n-1}\rfloor)$, $n\in\N$, $t\ge 0$, are the random maps introduced in Proposition \ref{Prop:GW_summary}. We have that
\begin{equation}\label{eq:nu_recursion}
\nu_{n}[0,t]\ =\ \nu_{n+1}[0,g_{n+1}(t)]
\end{equation}
for $n\in\N_{0}$ and $t\ge 0$, and this shows that $\nu_{n+1}$ differs from $\nu_{n}$ by a random perturbation of time. But the latter is negligible for large $n$ by the strong law of large numbers, viz.
$$
g_{n+1}(t)\ \stackrel{n\to\infty}{\longrightarrow}\ t\quad\text{a.s.},
$$
cf. Lemma \ref{summable_lemma}. For arbitrary fixed $k\in\N_{0}$, iteration of \eqref{eq:nu_recursion} provides us with
\begin{equation}\label{eq:nu_{k}_recursion}
\nu_{k}[0,t]\ =\ \nu_{n+k}[0,g_{n+k}\circ\cdots\circ g_{k+1}(t)],
\end{equation}
in particular
\begin{equation}\label{eq:nu_{0}_recursion}
\nu_{0}[0,\,t]\ =\ \nu_{n}[0,\wh{Z}_{n}(\lfloor t\rfloor)]
\end{equation}
for $n\in\N$ and $t\ge 0$. Since
$$
\lim_{t\to\infty}\lim_{n\to\infty}\wh{Z}_{n}(\lfloor t\rfloor)\ =\ \lim_{t\to\infty}Z_{\infty}(t)\ =\ \infty,
$$
equation \eqref{eq:nu_{0}_recursion} implies that, for each $T>0$,
\begin{equation}\label{eq:sup_{n}u_finite}
\sup_{n\ge 0}\,\nu_{n}[0,T]\ <\ \infty\quad\text{a.s.}
\end{equation}
Hence, $(\nu_{n})_{n\ge 0}$ is a.s. relatively compact in the vague topology (see 15.7.5 in \cite{Kallenberg:83}). Let $(\nu_{m_{n}})_{n\ge 1}$, where $(m_{n})_{n\in\N}$ is random, be an a.s. vaguely convergent subsequence and $\nu_{\infty}'$ its limit. From \eqref{eq:nu_{k}_recursion}, we have for every fixed $k\in\N_{0}$ and $m_{n}>k$ that
\begin{equation}\label{proof_m_{n}}
\nu_{k}[0,t]\ =\ \nu_{m_{n}}[0,g_{m_{n}}\circ\cdots\circ g_{k+1}(t)], \quad t\ge 0.
\end{equation}
By part (A2) of Proposition \ref{Prop:GW_summary}
$$
g_{m_{n}}\circ\cdots\circ g_{k+1}(t)\ \stackrel{n\to\infty}{\longrightarrow}\ \mu^{-k}Z_{k,\infty}(t\mu^{k})\quad\text{a.s.}
$$
in the space $D$ endowed with the $J_1$-topology. Sending $n\to\infty$ in \eqref{proof_m_{n}} and applying Lemma \ref{lemma_continuity}, we obtain that a.s.\
\begin{equation}\label{eq:nu_infty_well_defined}
\nu_{k}[0,t]\ =\ \nu_{\infty}'[0,\mu^{-k}Z_{k,\infty}(t\mu^{k})],\quad t\ge 0
\end{equation}
for every fixed $k\in\N_{0}$. By part (A5) of Proposition \ref{Prop:GW_summary}, the random set
$$ \cS\ :=\ \left\{\mu^{-k}Z_{k,\infty}(t\mu^{k}):t\ge 0,k\in\N_{0}\right\} $$
is a.s. dense in $[0,\infty)$. If $\nu_{\infty}''$ is another subsequential limit of $(\nu_{n})_{n\in\N_{0}}$, then
$$
\nu_{\infty}''[0,t]\ =\ \nu_{\infty}'[0,t]
$$
for all $t\in\cS$, and therefore $\nu''_{\infty}=\nu'_{\infty}$ a.s., proving \eqref{eq:vague_as_convergence}.

\vspace{.1cm}
It remains to show that the random process $G$ satisfies the restricted self-similarity property \eqref{G_selfsimilar}. But this follows immediately from 
$$
\nu_n[0,\mu t]\ =\ \mu^{-n}X^{(n)}_{\lfloor t\mu^{n+1}\rfloor}\ \eqdist\ \mu^{-n}X^{(n+1)}_{\lfloor t\mu^{n+1}\rfloor}\ =\ \mu\nu_{n+1}[0,t],\quad t\ge\mu^{-(n+1)},
$$
where the equality in law is a consequence of the stationarity of $(V_k)_{k\in\Z}$.\qed
\end{proof}

\section{Proofs in the infinite-mean case}\label{sec:infinite mean}

\subsection{Some auxiliary results about infinite-mean Galton-Watson processes}	
\begin{Lemma}\label{moments_bound_inf_mean}
Let $\theta$ be an $\N$-valued random variable satisfying Davies' assumption, viz.
$$
x^{-\alpha-\gamma(x)}\le \Prob\{\theta>x\}\le x^{-\alpha+\gamma(x)},\quad x\ge x_{0},
$$
for some $x_{0}>0$, $\alpha\in(0,1)$, and a nonincreasing nonnegative function $\gamma(x)$ such that $x^{\gamma(x)}$ is nondecreasing and
$\int_{x_{0}}^{\infty}\gamma(\exp(e^{x}))\,dx<\infty$. Let $(\theta_{n})_{n\in\N}$ be a sequence of independent copies of $\theta$, and put
$$
R(0):=0,\quad R(n):=\theta_{1}+\ldots+\theta_{n},\quad M(n):=\max_{k=1,\ldots,n} \theta_{k},\quad n\in\N.
$$
Then, for each $\varepsilon>0$ and each $\beta\in(0,\alpha)$, there exist $c,C>0$ such that, for all $n\in\N$,
\begin{equation}\label{eq:moments_bound_inf_mean}
cn^{-(1/\alpha+\varepsilon)\beta\alpha^{-1}\gamma(n^{1/\alpha+\varepsilon})}\ \le\ \Erw\left(\frac{ R(n)}{n^{1/\alpha}}\right)^{\beta}\ \le\ Cn^{(1/\alpha+\varepsilon)\beta\alpha^{-1}\gamma(n^{1/\alpha+\varepsilon})},
\end{equation}
and
\begin{equation}\label{eq:moments_bound_inf_mean2}
cn^{-(1/\alpha+\varepsilon)\beta\alpha^{-1}\gamma(n^{1/\alpha+\varepsilon})}\ \le\ \Erw\left(\frac{M(n)}{n^{1/\alpha}}\right)^{-\beta}\ \le\ Cn^{(1/\alpha+\varepsilon)\beta\alpha^{-1}\gamma(n^{1/\alpha+\varepsilon})}.
\end{equation}
\end{Lemma}
\begin{proof}
If $\theta$ belongs to the domain of attraction of an $\alpha$-stable law, $\alpha\in(0,1)$, equivalently if $x\mapsto \Prob\{\theta>x\}$ is regularly varying at infinity with index $\alpha$, then the $\beta$-th moment of $R(n)/c(n)$, where $c(\cdot)$ is such that $\lim_{n\to\infty}n\,\Prob\{\theta>c(n)\}=1$, converges to the $\beta$-th moment of the limit $\alpha$-stable law for all $\beta\in (0,\alpha)$. Unfortunately, Davies' condition does not imply that $\theta$ is in the domain of attraction of an $\alpha$-stable law. Yet, in some sense $\theta$ can be bounded from below and above by random variables with regularly varying tails, which is the idea employed in the following argument.

We first recall, see \cite[Lemma 3]{Davies:78}, that the function $x\mapsto x^{\gamma(x)}$ is slowly varying at infinity. Pick $x_{1}>x_{0}$ so large that $x^{-\alpha+\gamma(x)}\le 1$ for $x\ge x_{1}$ and let $\overline{\theta}\ge 1$ and $\underline{\theta}\ge 1$ be random variables
with distributions
\begin{align*}
&\Prob\{\overline{\theta}>x\}\ :=\
\begin{cases}
\hfill 1, &\text{if }1\le x < x_{1},\\
\sup_{y>x}y^{-\alpha+\gamma(y)},&\text{if }x\ge x_{1};
\end{cases}
\shortintertext{and}
&\Prob\{\underline{\theta}>x\}\ :=\
\begin{cases}
\hfill \Prob\{\theta>x\}, &\text{if }1\le x< x_{1},\\
\inf_{x_1\leq y\leq x}y^{-\alpha-\gamma(y)},&\text{if }x\ge x_{1}.
\end{cases}
\end{align*}
It is clear from the construction that, with $\le_{st}$ denoting stochastic majorization,
$$
\underline{\theta}\,\le_{st}\,\theta\,\le_{st}\,\overline{\theta}.
$$
Moreover, by Theorem 1.5.3 in \cite{BingGolTeug:89}, both $x\mapsto \Prob\{\overline{\theta}>x\}$ and $x\mapsto \Prob\{\underline{\theta}>x\}$ are regularly varying at infinity of order $-\alpha$ and therefore belong to the domain of attraction of an $\alpha$-stable law.

\vspace{.1cm}
Let $(\overline{\theta}_{k})_{k\ge 1}$ and $(\underline{\theta}_{k})_{k\ge 1}$ be sequences of independent copies of $\overline{\theta}$ and $\underline{\theta}$, respectively, with associated zero-delayed random walks $(\ovl{R}(k))_{k\ge 0}$ and $(\uline{R}(k))_{k\ge 0}$. Further, let $(\overline{c}(n))_{n\in\N}$ and $(\underline{c}(n))_{n\in\N}$ be such that
$$
\lim_{n\to\infty}n\,\Prob\{\underline{\theta}>\underline{c}(n)\}\ =\ \lim_{n\to\infty}n\,\Prob\{\overline{\theta}>\overline{c}(n)\}\ =\ 1.
$$
From Lemma 5.2.2 in \cite{Ibragimov+Linnik:71}, we infer
$$
0\ <\ \lim_{n\to\infty}\frac{\Erw\underline{R}(n)^{\beta}}{\underline{c}^{\beta}(n)}\ <\ \infty\quad\text{and}\quad 0\ <\ \lim_{n\to\infty}\frac{\Erw\overline{R}(n)^{\beta}}{\overline{c}^{\beta}(n)}\ <\ \infty,
$$
and since
$$
\Erw\underline{R}(n)^{\beta}\ \le\ \Erw R(n)^{\beta}\ \le\ \Erw\overline{R}(n)^{\beta},
$$
relation \eqref{eq:moments_bound_inf_mean} follows if we can show that, for some $c_{1},C_{1}>0$ and all $n\in\N$,
\begin{equation}\label{eq:moments_bound_aux1}
c_{1} n^{-(1/\alpha+\varepsilon)\gamma(n^{1/\alpha+\varepsilon})/\alpha}\ \le\  \frac{\underline{c}(n)}{n^{1/\alpha}}\quad\text{and}\quad \frac{\overline{c}(n)}{n^{1/\alpha}}\ \le\ C_{1} n^{(1/\alpha+\varepsilon)\gamma(n^{1/\alpha+\varepsilon})/\alpha}.
\end{equation}
We prove only the first inequality in \eqref{eq:moments_bound_inf_mean}, for the second one follows in a similar manner. It is known that $(\underline{c}(n))$ is regularly varying with index $1/\alpha$. Hence, for large enough $n$, we have $\underline{c}(n)\le n^{1/\alpha+\varepsilon}$. On the other hand, using the monotonicity of $x\mapsto x^{\gamma(x)}$, we have
$$
n\,\Prob\{\underline{\theta}>\underline{c}(n)\}\ \simeq\ n\,\underline{c}(n)^{-\alpha-\gamma(\underline{c}(n))}\ \ge\ (n^{-1/\alpha}\underline{c}(n))^{-\alpha}(n^{1/\alpha+\varepsilon})^{-\gamma(n^{1/\alpha+\varepsilon})},
$$
and therefore
$$
\limsup_{n\to\infty}\,(n^{-1/\alpha}\underline{c}(n))^{-\alpha}(n^{1/\alpha+\varepsilon})^{-\gamma(n^{1/\alpha+\varepsilon})}\ \le\ 1,
$$
yielding
$$
c_{1} n^{-(1/\alpha+\varepsilon)\gamma(n^{1/\alpha+\varepsilon})/\alpha}\ \le\  \frac{\underline{c}(n)}{n^{1/\alpha}}
$$
for some $c_{1}>0$.

\vspace{.1cm}
To show \eqref{eq:moments_bound_inf_mean2}, we argue in a similar manner. Set
$$
\overline{M}(n):=\max_{1\le k\le n}(-\overline{\theta}^{-1}_{k}),\quad\underline{M}(n):=\max_{1\le k\le n}(-\underline{\theta}_{k}^{-1}),\quad n\in\N.
$$
The Fisher-Tippett-Gnedenko theorem implies that $(-\overline{c}(n)\overline{M}(n))$ and $(-\underline{c}(n)\underline{M}(n))$ both converge weakly to a Weibull distribution. Moreover, Theorem 2.1 in \cite{Pickands:68} ensures that the moments of order $\beta$ also converge, so
$$
0\ <\ \lim_{n\to\infty}\Erw\Big(-\overline{c}(n)\overline{M}(n)\Big)^{\beta}\ <\ \infty\quad\text{and}\quad 0\ <\ \lim_{n\to\infty}\Erw\Big(-\underline{c}(n)\underline{M}(n)\Big)^{\beta}\ <\ \infty.
$$
On the other hand,
\begin{align*}
\Erw\Big(-\underline{c}(n)\underline{M}(n)\Big)^{\beta}\ &=\ \Erw\Big(\underline{c}(n)\min_{1\le k\le n}\underline{\theta}_{k}^{-1}\Big)^{\beta}\\
&=\ \Erw\left(\frac{\underline{c}(n)}{\max_{1\le k\le n}\underline{\theta}_{k}}\right)^{\beta}\ \ge\ \Erw\left( \frac{\underline{c}(n)}{\max_{1\le k\le n}\theta_{k}}\right)^{\beta}.
\end{align*}
and
\begin{align*}
\Erw\Big(-\overline{c}(n)\overline{M}(n)\Big)^{\beta}\ &=\ \Erw\Big(\overline{c}(n)\min_{1\le k\le n}\overline{\theta}_{k}^{-1}\Big)^{\beta}\\
&=\ \Erw\left(\frac{\overline{c}(n)}{\max_{1\le k\le n}\overline{\theta}_{k}}\right)^{\beta}\ \le\ \Erw\left(\frac{\overline{c}(n)}{\max_{1\le k\le n}\theta_{k}}\right)^{\beta}.
\end{align*}
Combining this with \eqref{eq:moments_bound_aux1}, we obtain \eqref{eq:moments_bound_inf_mean2}.\qed
\end{proof}

The next result is the counterpart of Lemma \ref{summable_lemma} in the infinite-mean case.

\begin{Lemma}\label{summable_lemma_inf_mean}
Let $(\theta^{(n)}_{k})_{n\in\N,\,k\in\N}$ be an array of independent copies of a positive random variable $\theta$ which satisfies the assumptions of Lemma \ref{moments_bound_inf_mean}. For each $n\in\N$, define the increasing random walk
$$
R^{(n)}(0):=0,\quad R^{(n)}(k):=\theta^{(n)}_{1}+\ldots+\theta^{(n)}_{k},\quad k\in\N.
$$
Then, for arbitrary $T>0$,
$$
\sum_{n=1}^{\infty}\sup_{t\in[0,\,T]}\Big|\alpha^n\log R^{(n)}(\lfloor  e^{t\alpha^{-(n-1)}}\rfloor)-t\Big|<\infty\quad\text{a.s.}
$$
\end{Lemma}

\begin{proof}
Put $m_{n}(t):=\lfloor  e^{t\alpha^{-(n-1)}}\rfloor$ and $N_{n}:=m_{n}(T)$. Then
$$
\alpha^{n-1}\log m_{n}(t)\ \le\ t\ <\ \alpha^{n-1}\log(m_{n}(t)+1)
$$
and therefore
\begin{align*}
&\hspace{-.5cm}\sup_{t\in[0,\,T]}\Big|\alpha^n\log R^{(n)}(\lfloor  e^{t\alpha^{-(n-1)}}\rfloor)-t\Big|\\
&\le\ \sup_{t\in[0,\,T]}\Big|\alpha^n\log R^{(n)}(m_{n}(t))-\alpha^n\log (m_{n}(t))^{1/\alpha}\Big|\ +\ \sup_{t\in[0,\,T]}\Big|\alpha^{n-1}\log m_{n}(t)-t\Big|\\
&\le\ \alpha^n\sup_{1\le k\le N_{n}}\left|\log \frac{R^{(n)}(k)}{k^{1/\alpha}}\right|\ +\ \alpha^{n-1}\log (1+(m_{n}(t))^{-1})\\
&\le\ \alpha^n\sup_{1\le k\le N_{n}}\left|\log\frac{R^{(n)}(k)}{k^{1/\alpha}}\right|\ +\ \alpha^{n-1}\log 2.
\end{align*}
Hence, we must show
\begin{equation*}
\sum_{n=1}^{\infty}\alpha^n\sup_{1\le k\le N_{n}}\Big|\log \frac{R^{(n)}(k)}{k^{1/\alpha}}\Big|\ <\ \infty\quad\text{a.s.}
\end{equation*}
which amounts to checking the following two relations:
\begin{align}
&\sum_{n=1}^{\infty}\alpha^n\sup_{1\le k\le N_{n}}\Big(\log^{+} \frac{R^{(n)}(k)}{k^{1/\alpha}}\Big)\ <\ \infty\quad\text{a.s.}\label{eq1:summable_lemma_inf_mean2}
\shortintertext{and}
&\sum_{n=1}^{\infty}\alpha^n\sup_{1\le k\le N_{n}}\Big(\log^{-}\frac{R^{(n)}(k)}{k^{1/\alpha}}\Big)\ <\ \infty\quad\text{a.s.}\label{eq2:summable_lemma_inf_mean2}
\end{align}
Fixing $\beta\in(0,\alpha)$, we obtain
\begin{align*}
\beta\,\Erw\left(\sup_{1\le k\le N_{n}}\log^{+} \frac{R^{(n)}(k)}{k^{1/\alpha}}\right)\ &=\ \Erw\left(\log \sup_{1\le k\le N_{n}}\frac{(R^{(n)}(k))^{\beta}}{k^{\beta/\alpha}}\right)\\
&\le\ \Erw\left(\log \sum_{j=0}^{\lfloor \log N_{n}\rfloor }\sup_{k\in (e^{-(j+1)}N_{n},\,e^{-j} N_{n}]}\frac{(R^{(n)}(k))^{\beta}}{k^{\beta/\alpha}}\right)\\
&\le\ \log \sum_{j=0}^{\lfloor T\alpha^{-(n-1)}\rfloor}\Erw\left(\sup_{k\in (e^{-(j+1)}N_{n},\,e^{-j} N_{n}]}\frac{(R^{(n)}(k))^{\beta}}{k^{\beta/\alpha}}\right)\\
\end{align*}
where the last line follows from Jensen's inequality. For any $\varepsilon>0$, we further infer with the help of Lemma \ref{moments_bound_inf_mean} and \eqref{eq:moments_bound_inf_mean}
\begin{align*}
\Erw\left(\sup_{k\in (e^{-(j+1)}N_{n},\,e^{-j} N_{n}]}\frac{R^{(n)}(k)^{\beta}}{k^{\beta/\alpha}}\right)\ &\le\ C\,\Erw\left(\frac{R^{(n)}(\lceil e^{-j} N_{n}\rceil)}{\lceil e^{-j} N_{n}\rceil^{1/\alpha}}\right)^{\beta}\\
&\leq\ C\Big(\lceil e^{-j}N_{n}\rceil^{(1/\alpha+\varepsilon)\gamma(\lceil e^{-j} N_{n}\rceil^{1/\alpha+\varepsilon})}\Big)^{\beta/\alpha},
\end{align*}
where $C\in (0,\infty)$ denotes a suitable constant which here and hereafter may differ from line to line. By combining the previous estimates and using the monotonicity of $x\mapsto x^{\gamma(x)}$, we obtain
\begin{align*}
\beta\,\Erw\left(\sup_{1\le k\le N_{n}}\log\frac{R^{(n)}(k)}{k^{1/\alpha}}\right)\ &\le\ C\ +\ \log \sum_{j=0}^{\lceil T\alpha^{-(n-1)}\rceil}\Big(\lceil e^{-j}N_{n}\rceil^{(1/\alpha+\varepsilon)\gamma(\lceil e^{-j} N_{n}\rceil^{1/\alpha+\varepsilon})}\Big)^{\beta/\alpha}\\
&\le\ C\ +\ \log\left(\Big(\lceil T\alpha^{-(n-1)}\rceil+1\Big)\Big(N_{n}^{(1/\alpha+\varepsilon)\gamma(N_{n}^{1/\alpha+\varepsilon})}\Big)^{\beta/\alpha}\right).
\end{align*}
Consequently, \eqref{eq1:summable_lemma_inf_mean2} follows from the inequality
$$
\sum_{n=1}^{\infty}\alpha^n \log \Big(N_{n}^{(1/\alpha+\varepsilon)\gamma(N_{n}^{1/\alpha+\varepsilon})}\Big)\ \le\ C\sum_{n=1}^{\infty}\gamma(N_{n}^{1/\alpha+\varepsilon})
$$
and the fact that $\int_{x_{0}}^{\infty}\gamma(\exp(e^{x}))\,dx<\infty$ implies (see calculations on p.~473 in \cite{Davies:78})
$$
\sum_{n=1}^{\infty}\gamma(e^{u\alpha^{-n}})\ <\ \infty.
$$
for any $u>0$.

\vspace{.1cm}
Equation \eqref{eq2:summable_lemma_inf_mean2} is verified along similar lines. Using $\log^{-}x=\log(x^{-1}\wedge 1)\le\log (1+x^{-1})$, we obtain for arbitrary $\beta\in(0,\alpha)$
\begin{align*}
\beta\,\Erw\left(\sup_{1\le k\le N_{n}}\log^{-} \frac{R^{(n)}(k)}{k^{1/\alpha}}\right)\ &\le\ \beta\,\Erw\left(\log\Bigg(1+\sup_{1\le k\le N_{n}}\Bigg(\frac{k^{1/\alpha}}{R^{(n)}(k)}\Bigg)\Bigg)\right)\\
&\le\ \Erw\log\Bigg(1+\sum_{j=0}^{\lfloor \log N_{n}\rfloor} \sup_{k\in (e^{-(j+1)}N_{n},\,e^{-j} N_{n}]}\Bigg(\frac{k^{\beta/\alpha}}{R^{(n)}(k)^{\beta}}\Bigg)\Bigg)\\
&\le\ \log\Bigg(1+\sum_{j=0}^{\lfloor\log N_{n}\rfloor}\Erw\Bigg(\sup_{k\in (e^{-(j+1)}N_{n},\,e^{-j} N_{n}]}\Bigg(\frac{k^{\beta/\alpha}}{R^{(n)}(k)^{\beta}}\Bigg)\Bigg)\Bigg)\\
&\le\ \log\Bigg(1+\sum_{j=0}^{\lfloor \log N_{n}\rfloor}\Erw\Bigg(\frac{\lceil e^{-j} N_{n}\rceil^{\beta/\alpha}}{(R^{(n)}(\lceil e^{-(j+1)}N_{n}\rceil))^{\beta}}\Bigg)\Bigg)\\
&\le\ \log\Bigg(1+\sum_{j=0}^{\lfloor \log N_{n}\rfloor } \Erw\Bigg(\frac{\lceil e^{-j} N_{n}\rceil^{1/\alpha}}{\max_{k=1,\ldots,\lceil e^{-(j+1)}N_{n}\rceil}\theta^{(n)}_{k}}\Bigg)^{\beta}\Bigg).
\end{align*}
By \eqref{eq:moments_bound_inf_mean2} in Lemma \ref{moments_bound_inf_mean},
$$
\Erw\left(\frac{\lceil e^{-j} N_{n}\rceil^{1/\alpha}}{\max_{k=1,\ldots,\lceil e^{-(j+1)}N_{n}\rceil}\theta^{(n)}_{k}}\right)^{\beta}\ \le\ C\Big(\lceil e^{-(j+1)}N_{n}\rceil^{(1/\alpha+\varepsilon)\gamma(\lceil e^{-(j+1)} N_{n}\rceil^{1/\alpha+\varepsilon})}\Big)^{\beta/\alpha},
$$
and this implies
$$
\sum_{n=1}^{\infty}\alpha^n\Erw\sup_{1\le k\le N_{n}}\Big(\log^{-}\frac{R^{(n)}(k)}{k^{1/\alpha}}\Big)<\infty
$$
by the same argument as above. The proof of Lemma \ref{summable_lemma_inf_mean} is complete.\qed
\end{proof}

The counterpart of Proposition \ref{Prop:GW_summary} is next.

\begin{Prop}\label{Prop:GW_summary_Davies}
Let $(R^{(n)}(k))_{k\in\N_{0},n\in\N}$ be as in Lemma \ref{summable_lemma_inf_mean} and put
$$ h_{n}(t)\ :=\ \alpha^{n}\log R^{(n)}(\lfloor  e^{t\alpha^{-(n-1)}}\rfloor) $$
for $n\in\N$ and $t\ge 0$.
\begin{description}[(B5)]\itemsep2pt
\item[(B1)] There exists a $D$-valued random process $(Z_{\infty}^{*}(t))_{t\ge 0}$ such that
$$
\big(h_{n}\circ \cdots\circ h_{1}(t)\big)_{t\ge 0}\ \stackrel{n\to\infty}{\longrightarrow}\ \big(Z_{\infty}^{*}(t)\big)_{t\ge 0}\quad\text{a.s.}
$$
in the space $D$ endowed with the $J_1$-topology. The process $(Z_{\infty}^{*}(t))_{t\ge 0}$ is the a.s.\ limit of $(\alpha^{n}\log Z_{n}(t))_{t\ge 0}$ as $n\to\infty$, where $Z_{n}(t)$ denotes the number of descendants in generation $n$ of ancestors $1,\ldots, \lfloor e^t\rfloor$ in a GWP with generic offspring variable $\theta$ and countably many ancestors $1,2,\ldots$ in generation $0$. The process $(Z_{\infty}^{*}(t))_{t\ge 0}$ has the following representation:
$$
Z_{\infty}^{*}(t)\ :=\ Z_{\infty}^{(*,1)}\vee Z_{\infty}^{(*,2)}\vee \ldots\vee Z_{\infty}^{(*,\lfloor  e^{t}\rfloor)},\quad t\ge 0,
$$
with $(Z_{\infty}^{(*,j)})_{j\in\N}$ denoting independent copies of $Z_{\infty}^{*}$, the limit in \eqref{eq:Z_infty}.

\item[(B2)] For every fixed $k\in\N_{0}$ there exists a copy $(Z_{k,\infty}^{*}(t))_{t\ge 0}$ of the process $(Z_{\infty}^{*}(t))_{t\ge 0}$ such that
$$
\Big(h_{n}\circ \cdots\circ h_{k+1}(t)\Big)_{t\ge 0}\ \stackrel{n\to\infty}{\longrightarrow}\ \Big(\alpha^{k}Z_{k,\infty}^{*}(t\alpha^{-k})\Big)_{t\ge 0}\quad\text{a.s.}
$$
in the space $D$ endowed with the $J_1$-topology.

\item[(B3)] As $t\to\infty$,
$$
\inf_{k\in\N}\,h_{k}\circ\cdots\circ h_{1}(t)\ \to\ \infty\quad\text{a.s.}
$$

\item[(B4)] For any $T>0$,
$$
\sup_{k\in\N}\sup_{t\in[0,T]}\Big|h_{n+k}\circ\cdots\circ h_{n+1}(t)-t\Big|\ \stackrel{n\to\infty}{\longrightarrow}\ 0\quad\text{a.s.}
$$
\item[(B5)] The set $\big\{\alpha^{k} Z_{k,\infty}^{*}(t\alpha^{-k}):t\ge 0,k\in\N\big\}$ is almost surely dense in $[0,\infty)$.
\end{description}
\end{Prop}

\begin{proof}
(B1) The statement is equivalent to
\begin{equation*}
\left(\alpha^{n}\log R^{(n)} \circ \cdots \circ R^{(1)}(j)\right)_{j\in\N}\ \stackrel{n\to\infty}{\longrightarrow}\ \left(Z_{\infty}^{*}(j)\right)_{j\in\N}\quad\text{a.s.}
\end{equation*}
Introducing
\begin{align*}
Z_{n,1}
&:=
R^{(n)}\circ \ldots \circ R^{(1)}(1),\\
Z_{n,j}
&:=
R^{(n)}\circ \ldots \circ R^{(1)}(j) - R^{(n)}\circ \ldots \circ R^{(1)}(j-1),
\quad j=2,3,\ldots,
\end{align*}
we obtain independent GWP's $(Z_{n,1})_{n\in\N}, (Z_{n,2})_{n\in\N},\ldots$ with generic offspring variable $\theta$. By \eqref{eq:Z_infty}, the random variables
$$
Z_{\infty}^{(*,j)}\ :=\ \lim_{n\to\infty}\alpha^n \log Z_{n,j},\quad j\in\N
$$
exist a.s.\ in $(0,\infty)$ and are independent with the same distribution as $Z_\infty^{*}$. It follows that
\begin{align*}
\alpha^n \log R^{(n)}\circ \ldots \circ R^{(1)}(j)\ =\ \ & \alpha^n \log (Z_{n,1} + \ldots + Z_{n,j})\\
\ \stackrel{n\to\infty}{\longrightarrow}\ \ &Z_{\infty}^{(*,1)} \vee \ldots \vee Z_{\infty}^{(*,j)},
\end{align*}
for any $j\in\N$ which completes the proof of (B1).

\vspace{0.2cm}
(B2) This is an immediate consequence of (B1) and the identity
$$
h_{n}\circ \cdots\circ h_{k+1}(t)\ =\ \alpha^{k}\alpha^{(n-k)}R^{(n)} \circ \cdots \circ R^{(k+1)}(\lceil  e^{t\alpha^{-k}}\rceil),\quad t\ge 0
$$
valid for $n>k$.

\vspace{0.2cm}
(B3) Keeping the notation from (B1), we have
$$
h_{k}\circ\cdots\circ h_{1}(t)\ =\ \alpha^{k}\log\Bigg(\sum_{j=1}^{\lfloor e^{t}\rfloor}Z_{k,j}\Bigg),\quad t\ge 0,
$$
and hence for all $t\ge 0$ and arbitrary $n_{0}\in\N$
\begin{align*}
\inf_{k\in\N}h_{k}\circ\cdots\circ h_{1}(t)\ &\ge\ \inf_{k\in\N}\max_{1\le j\le\lfloor e^{t}\rfloor}(\alpha^{k}\log Z_{k,j})\\
&=\ \Big(\inf_{k>n_{0}}\max_{1\le j\le\lfloor e^{t}\rfloor}(\alpha^{k}\log Z_{k,j})\Big)\wedge \Big(\inf_{k\le n_{0}}\max_{1\le j\le\lfloor e^{t}\rfloor}(\alpha^{k}\log Z_{k,j})\Big)\\
&\ge\ \Big(\max_{1\le j\le\lfloor e^{t}\rfloor}\inf_{k>n_{0}}(\alpha^{k}\log Z_{k,j})\Big)\wedge \Big(\max_{1\le j\le\lfloor e^{t}\rfloor}\inf_{k\le n_{0}}(\alpha^{k}\log Z_{k,j})\Big)
\end{align*}
by the minimax inequality. Since the $\inf_{k\le n_{0}}(\alpha^{k}\log Z_{k,j})$, $j\in\N$, are iid with a law having unbounded support\footnote{this is an easy consequence of the fact that $\theta$ has unbounded support in view of $\Erw\theta=\infty$}, i.e.
\begin{equation*}
\Prob\left\{\inf_{k\le n_{0}}(\alpha^{k}\log Z_{k,1})>z\right\}\ >\ 0
\end{equation*}
for every $z>0$, we deduce that
$$
\max_{1\le j\le\lfloor e^{t}\rfloor}\inf_{k\le n_{0}}(\alpha^{k}\log Z_{k,j})\ \stackrel{t\to\infty}{\longrightarrow}\ \to \infty\quad\text{a.s.}
$$
By the same arguments,
$$
\max_{1\le j\le\lfloor e^{t}\rfloor}\inf_{k > n_{0}}(\alpha^{k}\log Z_{k,j})\ \stackrel{t\to\infty}{\longrightarrow}\ \to \infty\quad\text{a.s.}
$$
if we can show the existence of $n_{0}\in\N$ such that
\begin{equation}\label{eq:B4_proof}
\Prob\{\inf_{k\ge n_{0}}(\alpha^{k}\log Z_{k,1})>z\}>0
\end{equation}
for all $z>0$. To this end, fix $z>0$ and note that, by \cite[Theorem 2]{Davies:78}, we have $a:=\Prob\{Z_{\infty}^{*} > z+1\}<1$. On the other hand, we know from (B1) that
$$
\lim_{n\to\infty}\,\Prob\left\{\sup_{k\ge n}|\alpha^k\log Z_{k,1}-Z_{\infty}^{*}|\ge 1\right\}\ =\ 0.
$$
In particular, for each $0<\delta<a$, we find $n_{0}\in\N$ such that
$$
\Prob\left\{\sup_{k\ge n_{0}}|\alpha^k\log Z_{k,1}-Z_{\infty}^{*}|\ge 1\right\}\ <\ \delta,
$$
and so
\begin{align*}
\Prob\left\{\inf_{k\ge n_{0}}(\alpha^{k}\log Z_{k,1})\le z\right\}\ &=\ \Prob\left\{\inf_{k\ge n_{0}}(\alpha^{k}\log Z_{k,1})\le z,\,Z_{\infty}^{*} \le z+1\right\}\\
&+\ \Prob\left\{\inf_{k\ge n_{0}}(\alpha^{k}\log Z_{k,1})\le z,\,Z_{\infty}^{*} > z+1\right\}\\
&\le\ 1-a+\Prob\left\{\sup_{k\ge n_{0}}|\alpha^k\log Z_{k,1}-Z_{\infty}^{*}|\ge 1\right\}\\
&\le\ 1-a+\delta<1,
\end{align*}
which proves \eqref{eq:B4_proof}.

\vspace{0.2cm}
The proofs of (B4) and (B5) are omitted because they follow verbatim those of parts (A4) and (A5) of Proposition \ref{Prop:GW_summary} (with Lemma \ref{summable_lemma_inf_mean} for part (B4) instead of Lemma \ref{summable_lemma} for part (A4) in Proposition \ref{Prop:GW_summary}). \qed
\end{proof}

\subsection{Proof of Theorems \ref{main1_for_RW_heavy}, \ref{main1_for_RW_heavy_N}, \ref{main1_for_RW_heavy_T} and \ref{fixed_point_solutions_inf_mean}, and Proposition \ref{Prop_Sibuya}}

Define the random map $\phi:\R^{\N}\to\R^{\N}$ by
\begin{equation}\label{eq:phi_mapping_definition}
\phi((x_{1},x_2,\ldots))\ :=\ \alpha(x_{R(1)},x_{R(2)},\ldots)
\end{equation}
and further $\phi_{n}:\R^{\N}\to\R^{\N}$ for $n\in\N$ by
\begin{equation}\label{eq:phi_seq_mapping_definition}
\phi_{n}((x_{1},x_2,\ldots))\ :=\ \alpha (x_{R^{(n)}(1)},x_{R^{(n)}(2)},\ldots),\quad n\in\N
\end{equation}
which are independent copies of $\phi$ (compare \eqref{eq:psi_mapping_definition} and \eqref{eq:psi_seq_mapping_definition} in the finite-mean case).

\begin{proof}[of Theorem \ref{main1_for_RW_heavy}]
The convergence in \eqref{eq:RW_S_conv_heavy} is a consequence of \eqref{eq:basic_identity} and part (B1) of Proposition \ref{Prop:GW_summary_Davies}. The SFPE \eqref{fixed_point_inf_mean} follows from the continuity of the map $\phi$ with respect to the product topology on $\R^{\N}$.\qed
\end{proof}

\begin{proof}[of Theorem \ref{main1_for_RW_heavy_N}]
Note first that  $(N^{(n)}_{\lfloor\exp(x\alpha^{-n})\rfloor})_{x\in\R^{+}}$ is the counting process associated with the increasing random sequence $(\alpha^{n} \log S^{(n)}_{k})_{k\ge 1}$. By Theorem \ref{main1_for_RW_heavy},  the finite-dimensional distributions of this sequence converge as $n\to\infty$ to those of the nondecreasing random sequence  $(Z_{\infty}^{(*,1)} \vee \ldots \vee Z_\infty^{(*,k)})_{k\ge 1}$ with associated counting process $(N''(x))_{x\in\R^{+}}$. Also, for every fixed $n\in\N$,
$$
\lim_{k\to\infty} \alpha^n \log S_{k}^{(n)}\ =\ \lim_{k\to\infty} (Z_{\infty}^{(*,1)} \vee \ldots \vee Z_\infty^{(*,k)})\ =\ \infty \quad \text{a.s.}
$$
because $\xi>0$ a.s.\  and $Z_\infty^{*}$ is not bounded from above. The latter statement follows from the SFPE $\alpha(Z_{\infty}^{(*,1)} \vee \ldots \vee Z_\infty^{(*,\xi)})\eqdist Z_\infty^*$.  By Lemma \ref{lem:skorokhod_counting} from the Appendix, this implies the weak convergence of the corresponding counting processes on $D$ endowed with the $M_1$-topology.
\end{proof}

\begin{proof}[of Theorem \ref{main1_for_RW_heavy_T}]
Here the assertion is implied by
\begin{align*}
\Prob\left\{T(\lfloor e^{\alpha^{-n}x}\rfloor)-n\le k\right\}\ &=\ \Prob\left\{S_{1}^{(n+k)}>\lfloor e^{\alpha^{-n}x}\rfloor\right\}\\
&=\ \Prob\left\{\alpha^{n+k}\log S_{1}^{(n+k)}>\alpha^{k}x\right\}\ \stackrel{n\to\infty}\longrightarrow\ \Prob\{Z_{\infty}^{*}>\alpha^{k}x\}.\quad\qed
\end{align*}
\end{proof}

\begin{proof}[of Proposition \ref{Prop_Sibuya}]
If $\xi$ has a Sibuya-distribution with parameter $\alpha$, then the size $Z_{n}$ of the associated GWP in generation $n$ has a Sibuya distribution with parameter $\alpha^n$ because $f_{\alpha}  \circ f_{\beta} = f_{\alpha\beta}$ and hence,
$$
\Erw t^{Z_{n}}\ =\ f_{\alpha}\circ \ldots \circ f_{\alpha}(t)\ =\ f_{\alpha^n}(t). 
$$
By Lemma~\ref{lem:sibuya_null} in Appendix, the random variable $Z^{*}_\infty$ from~\eqref{eq:Z_infty} has a standard exponential law. It remains to argue that
$$
\sum_{j=1}^{\infty} \delta_{Z_{\infty}^{(*,1)}\vee \ldots \vee Z_{\infty}^{(*,j)}}
\ \eqdist\ \sum_{i=1}^{\infty} G_{i}\delta_{P_{i}}.
$$
But a well-known consequence of the memoryless property of the exponential distribution is that the record values (taken without repetitions and denoted by $P_{1},P_2,\ldots$) in the sequence $(Z_{\infty}^{(*,j)})_{j\ge 1}$ form a Poisson point process (Tata's representation, \cite[p.~69]{Nevzorov:01}). Given the record values $P_1,P_2,\ldots$, the interrecord times $G_1,G_2,\ldots$ are independent and have geometric distributions with parameters $e^{-P_{1}},e^{-P_{2}},\ldots$, respectively, by \cite[Theorem 17.1, p.~77]{Nevzorov:01}.
\qed
\end{proof}

\begin{proof}[of Theorem \ref{fixed_point_solutions_inf_mean}]
Again our arguments follow along similar lines as those in the proof of Theorem \ref{fixed_point_solutions_fin_mean} for the finite-mean case. Let $(X^{(0)}_{1},X^{(0)}_2,\ldots)$ be a solution to \eqref{fixed_point_inf_mean}, i.e.
$$
(X^{(0)}_{1},X^{(0)}_2,\ldots)\ \eqdist\ \alpha(X^{(0)}_{R(1)},X^{(0)}_{R(2)},\ldots).
$$
By the Kolmogorov consistency theorem, the underlying probability space $(\Omega,\fA,\Prob)$ may be assumed to be large enough to carry the following objects:
\begin{itemize}\itemsep2pt
\item the random sequence $(X^{(0)}_{1},X^{(0)}_{2},\ldots)$;
\item a two-sided sequence $(R^{(n)}(\cdot))_{n\in\Z}$ of independent copies of the random walk $R(\cdot)$ and the corresponding sequence of random maps \eqref{eq:phi_seq_mapping_definition};
\item a two-sided stationary sequence $(V_{k})_{k\in\Z}=\big((X_{1}^{(k)},X_{2}^{(k)},\ldots)\big)_{k\in\Z}$ such that $V_{k}$ is independent of $(\phi_{n})_{n\le k}$ for each $k\in\Z$, and
$$ V_{0}\ :=\ \big(X^{(0)}_{1},X^{(0)}_{2},\ldots\big)\quad\text{and}\quad V_{k}\ =\ \phi_{k+1}(V_{k+1}),\quad k\in\Z, $$
thus
$$ \big(X^{(k)}_{1},X^{(k)}_{2},\ldots\big)\ =\ \alpha\left(X^{(k+1)}_{R^{(k+1)}(1)},X^{(k+1)}_{R^{(k+1)}(2)},\ldots\right),\quad k\in\Z. $$
\end{itemize}
Define a sequence of random measures
$(\upsilon_{n})_{n\ge 0}$ on $[0,\infty)$ by
$$
\upsilon_{n}[0,t]\ :=\ \alpha^n \log X^{(n)}_{\lfloor\exp(t\alpha^{-n})\rfloor},\quad t \ge  0
$$
for $n\in\N_{0}$. As in the proof of Theorem \ref{fixed_point_solutions_fin_mean}, it is enough to show that the sequence $(\upsilon_{n})_{n\in\N}$ converges almost surely in the vague topology, i.e.
\begin{equation}\label{eq:vague_convergence_{k}appa}
\upsilon_{n}\ \vag\ \upsilon_{\infty}\quad\text{a.s.}
\end{equation}

Recall from Proposition \ref{Prop:GW_summary_Davies} that $h_{n}(t)=\alpha^{n}\log R^{(n)}(\lfloor  e^{t\alpha^{-(n-1)}}\rfloor)$ for $n\in\N$ and $t\ge 0$. We have
\begin{equation}\label{eq:kappa_recursion}
\upsilon_{n}[0,t]\ =\ \upsilon_{n+1}[0,h_{n+1}(t)]
\end{equation}
and upon iteration
\begin{equation}\label{eq:kappa_{k}_recursion}
\upsilon_{k}[0,\,t]\ =\ \upsilon_{n+k}[0,\,h_{n+k}\circ\cdots\circ h_{k+1}(t)]
\end{equation}
for $n\in\N_{0}$ and $t\ge 0$, in particular for $k=0$
\begin{equation}\label{eq:kappa_{0}_recursion}
\upsilon_{0}([0,\,t])\ =\ \upsilon_{n}\left[0,\,\alpha^{n}\log\left(\sum_{j=1}^{\lfloor  e^{t}\rfloor}Z_{n,j}\right)\right],\quad n\in\N,\quad t\ge 0.
\end{equation}
Since
$$
\lim_{t\to\infty}\lim_{n\to\infty}\alpha^{n}\log \sum_{j=1}^{\lfloor  e^{t}\rfloor}Z_{n,j}=\lim_{t\to\infty}Z_{\infty}^{*}(t)=\infty \quad \text{a.s.},
$$
see (B1) in Proposition \ref{Prop:GW_summary_Davies}, equation \eqref{eq:kappa_{0}_recursion} implies that, for any $T>0$,
$$
\sup_{n\in\N_{0}}\upsilon_{n}[0,T]\ <\ \infty\quad\text{a.s.}
$$
and therefore almost sure relative compactness of $(\upsilon_{n})_{n\ge 0}$ in the vague topology.

\vspace{.1cm}
Let $(\upsilon_{m_{n}})_{n\ge 1}$ be an a.s. vaguely convergent subsequence and $\upsilon'_{\infty}$ its limit. From \eqref{eq:kappa_{k}_recursion}, we have a.s. for each $k\in\N_{0}$ and $m_{n}>k$ that
\begin{equation}\label{proof_m_{n}_{k}appa}
\upsilon_{k}[0,t]\ =\ \upsilon_{m_{n}}[0,h_{m_{n}}\circ\cdots\circ h_{k+1}(t)],\quad t\ge 0.
\end{equation}
By part (B2) of Proposition \ref{Prop:GW_summary_Davies},
$$
h_{m_{n}}\circ\cdots\circ h_{k+1}(t)\ \stackrel{n\to\infty}{\longrightarrow}\ \alpha^{k}Z_{k,\infty}^{*}(t\alpha^{-k})\quad\text{a.s.}
$$
in the space $D$ endowed with the $J_1$-topology. Sending $n\to\infty$ in \eqref{proof_m_{n}_{k}appa} and applying Lemma \ref{lemma_continuity}, we obtain
$$
\upsilon_{k}[0,t]\ =\ \upsilon'_{\infty}[0,\,\alpha^{k}Z_{k,\infty}^{*}(t\alpha^{-k})],\quad t>0
$$
for every fixed $k\in\N_{0}$. By part (B6) of Proposition \ref{Prop:GW_summary_Davies}, the random set consisting of $\alpha^{k}Z_{k,\infty}^{*}(t\alpha^{-k})$ for $t\ge 0,\,k\in\N_{0}$ is a.s. dense in $[0,\infty)$. Hence, if $\upsilon''_{\infty}$ denotes another subsequential limit of $(\upsilon_{n})_{n\ge 1}$, then a.s.
$$
\upsilon''_{\infty}[0,t]\ =\ \upsilon'_{\infty}[0,t]
$$
on a dense subset of $[0,\infty)$ and therefore $\upsilon''_{\infty}\equiv\upsilon'_{\infty}$, proving the convergence \eqref{eq:vague_convergence_{k}appa}. 

\vspace{.1cm}
The restricted self-similarity property \eqref{G_selfsimilar2} can be checked as in the proof of Theorem~\ref{fixed_point_solutions_fin_mean}.\qed
\end{proof}

\section{Appendix}
\begin{Lemma}\label{lemma_continuity}
Let $\mathcal{M}:=\mathcal{M}[0,\infty)$ be the set of locally finite measures on $[0,\infty)$ endowed with the vague topology and let $\phi:\mathcal{M}\times [0,\infty)\to \R^{+}$ be defined by
$$ \phi(\mu,x)=\mu([0,x]). $$
Endowing $\mathcal{M}\times [0,\infty)$ with the product topology, the mapping $\phi$ is continuous at any $(\mu,x)$ such that $\mu(\{x\})=0$.
\end{Lemma}

\begin{proof}
Let $\mu_{n}\vag \mu_{0}$, $\lim_{n\to\infty}x_{n}=x_{0}$ and set $f_{0}(s):=\1_{\{s\le x_{0}\}}$, $B_{0}:=(0,\infty)\setminus\{x_{0}\}$, $f_{n}(s):=\1_{\{s\le x_{n}\}}$, $B_{n}:=[0,\infty)$ for $n\in\N$. The result now follows from Lemma 15.7.3 in \cite{Kallenberg:83}.\qed
\end{proof}

\begin{Lemma}\label{lem:skorokhod_counting}
For every $n\in\N\cup\{\infty\}$ let $0\le X_1^{(n)} \le X_2^{(n)}\le  \ldots$ be a random sequence such that $\lim_{k\to\infty} X_{k}^{(n)}=\infty$ a.s.\ and
\begin{equation}\label{eq:appendix_weak_seq}
(X_1^{(n)}, X_2^{(n)},\ldots)\ \todistrfd\ (X_1^{(\infty)}, X_2^{(\infty)}, \ldots).
\end{equation}
Define the corresponding counting processes $N^{(n)} (x) := \#\{k\in\N: X_{k}^{(n)}\le x\}$, $x\ge 0$, $n\in\N\cup\{\infty\}$. Then
$$
(N^{(n)}(x))_{x\ge 0}\ \todistrD\ (N^{(\infty)}(x))_{x\ge 0}
$$
in $D$ endowed with the $M_1$-topology, and the latter may be even replaced with the $J_1$-topology if the sequence $(X_{k}^{(\infty)})_{k\ge 1}$ is strictly increasing with probability $1$.
\end{Lemma}

\begin{proof}
Consider the space $L_{\le}$ of all sequences $y= (y_{k})_{k\in\N}$ such that $0\le y_{1}\le y_{2}\le \ldots$ and $\lim_{k\to\infty} y_{k}= \infty$. We endow $L_{\leq}$ with the topology of pointwise convergence. It is well-known that this topology is metrizable.  Consider a map $\Psi: L_\le \to D$ which assigns to each sequence $y = (y_{k})_{k\in\N}$ the corresponding counting function
$$
\Psi(y) (x) = \#\{k\in\N: y_{k} \le x\}, \quad x\ge 0.
$$
If we endow the $D$ with the $M_1$-topology, then it is an easy exercise to check that the map $\Psi$ is continuous on $L_{\le}$. Let us now endow $D$ with the $J_1$-topology. Then, $\Psi$ is not continuous on  $L_{\le}$ (because the points can build clusters in the limit), but it is continuous on the subset $L_{<}$ consisting of strictly increasing sequences.

We can consider $(X_{k}^{(n)})_{k\ge 1}$, $n\in\N\cup\{\infty\}$, as random elements with values in $L_{\leq}$ (in  $L_{<}$ when using the $J_1$-topology). Then the convergence in \eqref{eq:appendix_weak_seq} is equivalent to the weak convergence of the corresponding random elements. The statement of the lemma thus follows from the continuous mapping theorem.\qed
\end{proof}

\begin{Lemma}\label{lem:sibuya_null}
If $S_{\alpha}$ denotes a random variable with a Sibuya distribution with parameter $0<\alpha<1$, then
$$
\lim_{\alpha\downarrow 0} \Prob\{\alpha \log S_{\alpha} \le x\}\ =\ 1-e^{-x}
$$
for all $x>0$.
\end{Lemma}
\begin{Rem}
It is known \cite{Cressie:75} that if $Z_{\alpha}$ is a positive $\alpha$-stable variable with $\alpha\in (0,1)$, then
$\alpha \log Z_\alpha$ converges in distribution to the Gumbel double exponential law.
\end{Rem}
\begin{proof}[Proof of Lemma \ref{lem:sibuya_null}]
Consider independent Bernoulli variables $B_{1},B_2,\ldots$, such that $\Prob\{B_{i}=1\}=\alpha/i$. Then it is easy to check that $S_\alpha := \min \{n\in \N\colon B_{i} =1\}$ has the required Sibuya distribution with parameter $\alpha$. Taking any $x>0$, we now infer
$$
\Prob\{\alpha \log S_{\alpha} \le x\}\ =\ \Prob \{S_\alpha \le \lfloor e^{x/\alpha}\rfloor\}\ =\ 1-\Prob\{B_{1}+\ldots+B_{\lfloor e^{x/\alpha}\rfloor}=0\}.
$$
But for $\alpha\downarrow 0$, the probability on the right-hand side converges to $ e^{-x}$ because
$$
B_{1}+\ldots+B_{\lfloor e^{x/\alpha}\rfloor}\ \overset{d}{\underset{\alpha\downarrow 0}\longrightarrow}\ \textrm{Poi}\,(x).
$$
For the proof, just note that
$
\sum_{i=1}^{\lfloor e^{x/\alpha}\rfloor} \frac{\alpha}{i} \simeq \alpha \log \lfloor e^{x/\alpha}\rfloor \simeq x,
$
as $\alpha \downarrow 0$, and apply the Poisson limit theorem.\qed
\end{proof}

\acknowledgement{The authors would like to thank the referee for valuable comments and pointing out some inaccuracies in the first version of the paper. We also thank Prof. Alexander Iksanov for drawing our attention to the paper \cite{MaejimaSato:99}. The research of Gerold Alsmeyer was supported by the Deutsche Forschungsgemeinschaft (SFB 878), the research of Alexander Marynych by the Alexander von Humboldt Foundation.}

\bibliographystyle{abbrv}
\bibliography{StoPro}

\end{document}